\documentclass[12pt,oneside]{amsart}

 \newcommand\id{\mathrm{id}}
 \newcommand\tr{\operatorname{tr}}

 \newcommand\qRa{\quad\Rightarrow\quad}

 \newcommand\op{\oplus}

 \renewcommand\Im{\mathfrak{Im}}
 \renewcommand\Re{\mathfrak{Re}}

 \newcommand\fa{{\mathfrak a}} 
 
 \newcommand\fc{{\mathfrak c}}

 \newcommand\fe{{\mathfrak e}}

 \newcommand\fk{{\mathfrak k}}

 \newcommand\fn{{\mathfrak n}}

 \newcommand\fr{{\mathfrak r}}
 \newcommand\fs{{\mathfrak s}}
 \newcommand\fsl{\mathfrak{sl}}
 \newcommand\fso{\mathfrak{so}}

 \newcommand\fv{{\mathfrak v}}

 \newcommand\fX{{\mathfrak X}}

 \newcommand\cD{{\mathcal D}}
 \newcommand\cE{{\mathcal E}}

 \newcommand\cQ{{\mathcal Q}}

 \newcommand\sfv{\mathsf{v}}
 \newcommand\sfw{\mathsf{w}}

 \newcommand\sfH{\mathsf{H}}

 \newcommand\sfT{\mathsf{T}}

 \newcommand\sfX{\mathsf{X}}
 \newcommand\sfY{\mathsf{Y}}
 \newcommand\sfZ{\mathsf{Z}}

 \newcommand\bbA{{\mathbb A}}
 
 \newcommand\bbC{{\mathbb C}}

 \newcommand\bbH{{\mathbb H}}

 \newcommand\bbP{{\mathbb P}}
 
 \newcommand\bbR{{\mathbb R}}

 \newcommand\sgn{\mathrm{sgn}}

 \newcommand\tspan{\mathrm{span}}
 \newcommand\Sym{\mathrm{Sym}}
 
 \newcommand\Ben{\begin{enumerate}}
 \newcommand\Een{\end{enumerate}}
 \newcommand\Bex{\begin{example}}
 \newcommand\Eex{\end{example}}
 \newcommand\GL{\operatorname{GL}}

\def\tO{\text{O}}
 \newcommand\SL{\mathrm{SL}}

 \newcommand\SU{\text{SU}}

 \newcommand\fsu{\mathfrak{su}}
 \newcommand\ad{{\rm ad}}
 
 \newcommand\Ad{{\rm Ad}}
 \def\assoc/{associative}
 \def\arb/{arbitrary}
 \def\btw/{between}
 \def\coeff/{coefficient}
 \def\cohom/{cohomology}
 \def\coord/{coordinate}
 \def\coordsys/{coordinate system}
 \def\cpt/{compact}
 \def\cred/{completely reducible}
 \def\cts/{continuous}
 \def\dga/{differential-graded algebra}
 \def\dR/{de Rham}
 \def\Euc/{Euclidean} 
 \def\grp/{group}
 \def\hom/{homomorphism}
 \def\inv/{invariant}
 \def\iso/{isomorphism}
 \def\La/{Lie algebra}
 \def\Lag/{Lagrangian Grassmannian}
 \def\LG/{Lie group}
 \def\MA/{Monge--Amp\`ere}
 \def\MC/{Maurer--Cartan}
 \def\lintr/{linear transformation} 
 \def\mfld/{manifold}
 \def\nb/{normal bundle}
 \def\nbd/{neighbourhood}
 \def\nondeg/{non-degenerate}
 \def\posdef/{positive definite}
 \def\pu/{partition of unity}
 \def\rep/{representation}
 \def\Riem/{Riemannian}
 \def\sg/{subgroup}
 \def\ss/{semi-simple}
 \def\inv/{invariant}
 \def\irr/{irreducible}
 \def\Jacid/{Jacobi identity}
 \def\li/{linearly independent}
 \def\nd/{nowhere dependent}
 \def\nz/{nowhere zero}
 \def\on/{orthonormal}
 \def\onb/{\on/ basis}
 \def\orc/{\orth/ complement}
 \def\orth/{orthogonal}
 \def\orp/{\orth/ projection}
 \def\pde/{partial differential equation}
 \def\resp/{respectively}
 \def\seq/{sequence}
 \def\std/{standard}
 \def\SW/{Stiefel-Whitney}
 \def\uc/{universal cover}
 \def\vb/{vector bundle}
 \def\vf/{vector field}
 \def\vs/{vector space}
 \def\wrt/{with respect to}
 
 \renewcommand\mod{\,{\rm mod}\ }
 \renewcommand\dim{{\rm dim}}

\usepackage{marginnote}
\usepackage{framed}
\usepackage{amsmath,amsfonts,mathrsfs,amsthm,amssymb}
\usepackage{mathtools}
\usepackage{xcolor}

\newtheorem{theorem}{Theorem}[section]
\newtheorem{lemma}[theorem]{Lemma}
\newtheorem{cor}[theorem]{Corollary}
\newtheorem{prop}[theorem]{Proposition}
\theoremstyle{definition}
\newtheorem{defn}[theorem]{Definition}
\newtheorem{example}[theorem]{Example}

\theoremstyle{remark}
\newtheorem{remark}[theorem]{Remark}
\numberwithin{equation}{section}
\newcommand\diag{\operatorname{diag}}
\newcommand\bi{{\bf i}}
\newcommand\bj{{\bf j}}
\newcommand\bk{{\bf k}}
\newcommand{\spn}[1]{\langle#1\rangle}
\newcommand\Int{\operatorname{Int}}

\title{Homogeneous Levi non-degenerate hypersurfaces in $\bbC^3$}

\author{Boris Doubrov}
\address{Faculty of Mathematics and Mechanics, Belarusian State 
University, Nezavisimosti ave. 4, 220050, Minsk, Belarus}
\email{doubrov@bsu.by}
 
\author{Alexandr Medvedev}
\address{International School for Advanced Studies, via Bonomea 
265, Trieste	34136, Italy}
\email{amedvedev@sissa.it}
 
\author{Dennis The}
\address{Department of Mathematics and Statistics, University of Troms\o, 90-37, Norway; Fakult\"at f\"ur Mathematik, Universit\"at Wien, Oskar-Morgenstern-Platz 1, 1090 Wien, \"Osterreich}
\email{dennis.the@uit.no}


\begin{document}
\begin{abstract}
We classify all (locally) homogeneous Levi non-degenerate real hypersurfaces in $\bbC^3$ with symmetry algebra of dimension $\ge 6$.
\end{abstract}

\maketitle

\section{Introduction}

The main goal of this paper is to provide the complete (local) classification of \emph{multiply-transitive} (Levi non-degenerate) real hypersurfaces in $\bbC^3$, i.e.\ hypersurfaces with transitive symmetry algebra and stabilizer of dimension $\ge 1$. It is known~\cite{Tanaka1962,ChernMoser1974} that any real hypersurface in $\bbC^n$ with non-degenerate Levi form has symmetry algebra of dimension at most $n(n+2)$, which is achieved if and only if it is locally equivalent (under biholomorphic transformations) to a hyperquadric given by:
 \[
 \Im(w) = z_1\bar z_1 \pm \dots \pm z_{n-1}\bar z_{n-1}.
 \]
 In $\bbC^3$, the next possible dimension of the symmetry algebra for a Levi non-degenerate hypersurface is~$8$,  which is achieved for the so-called Winkelmann hypersurface~\cite{Wink1995}:
 \begin{align} \label{E:Wink4}
 \Im(w+\bar z_1 z_2) = |z_1|^4,
 \end{align}
where $(z_1,z_2,w)$ are holomorphic coordinates in $\bbC^3$. This real hypersurface has 8-dimensional symmetry algebra transitive on the hypersurface as well and on both open (unbounded) domains it splits $\bbC^3$.  We show in this paper that hyperquadrics and the Winkelmann hypersurface are the only homogeneous Levi non-degenerate hypersurfaces in $\bbC^3$ whose symmetry algebras have open orbits in $\bbC^3$.

The analogous classification result in $\bbC^2$ was obtained by \'Elie Cartan in his pioneering works~\cite{Cartan1932a,Cartan1932b} on this subject. He also used this result to prove that the only bounded homogeneous domain in $\bbC^2$ is the interior of a hypersphere. He claimed to prove a similar result for bounded homogeneous domains in $\bbC^3$, but
the details of this proof were never published and seem to be hidden in the archives of his notes\footnote{Private communication with Robert Bryant.}. This led him to believe~\cite{Cartan1935} that the only bounded homogeneous domains in $\bbC^n$ for any $n\ge 2$ are given by symmetric homogeneous spaces. However, this proved to be not correct, and the first counterexample was discovered by Piatetski-Shapiro in 1959~\cite{Piat-Sh1959}. 
 
Levi non-degenerate hypersurfaces in $\bbC^3$ with large symmetry algebras were extensively studied in a series of papers by A.V.~Loboda. He classified all Levi non-degenerate hypersurfaces with 7-dimensional symmetry algebra~\cite{Loboda2001a,Loboda2001b}, as well as all hypersurfaces with 6-dimensional symmetry algebra and positive definite Levi form~\cite{Loboda2003}. In our paper, we complete the classification of all multiply-transitive hypersurfaces in $\bbC^3$ by providing the full list of Levi indefinite hypersurfaces in $\bbC^3$ with 6-dimensional symmetry algebra. We also correct the Levi definite list \cite[Theorem 3]{Loboda2003} by adding one missing hypersurface with 6-dimensional symmetry algebra:
 \[
 v = \frac{x_2^2}{1+x_1} -\ln(1+x_1).
 \]
Here $z_1 = x_1+iy_1$, $z_2=x_2+iy_2$, $w=u+iv$ so this hypersurface is {\em tubular} (see Section~\ref{S:tubular} for the definition).  This corresponds to the Levi definite real form of case D.6-1 in~\cite{DMT2014} -- see also Table~\ref{F:affineD}. The symmetry algebra here is isomorphic to the semidirect product of $\fsl(2,\bbR)$ and the 3-dimensional Heisenberg algebra.  In Table~\ref{F:Loboda}, we match Loboda's classifications with our results.

The main idea of our classification approach is to pass from Levi non-degenerate hypersurfaces in $\bbC^n$ to their complex analogue, which turns out to be a complete system of 2nd order PDEs on one function of $(n-1)$ independent variables (see \S \ref{S:Complexification}). Such systems of PDEs have the same dimension for their symmetry algebra, which is multiply-transitive on the first jet-space $J^1(\bbC^{n-1},\bbC)$ if and only if the corresponding real hypersurface in $\bbC^n$ is multiply-transitive.

Geometrically, any Levi non-degenerate hypersurface $M\subset \bbC^n$ inherits a natural CR structure of codimension~1, which consists of a contact distribution $C\subset TM$ equipped with a complex structure $J\colon C\to C$. This complex structure is compatible with the natural conformal symplectic form on $C$ and integrable. Both these conditions are equivalent to the fact that the eigenspaces $J(i)$ and $J(-i)$ of the operator $J$ on the complexification $C^{\bbC}$ should be integrable subdistributions of the complexified contact distribution. 

The corresponding complete systems of second order PDEs are encoded as complex analytic manifolds of dimension $2n-1$ equipped with contact distribution $C$ decomposed into the direct sum of two completely integrable subdistributions $E$ and $V$.  Such structures are called \emph{integrable Legendrian contact (ILC) structures} and they were studied in detail in~\cite{DMT2014}.  For $n=3$, the fundamental invariant obstructing flatness\footnote{Flatness refers to being equivalent under point transformations to the trivial PDE system $u_{jk} = 0$.} is harmonic curvature $\kappa_H$, which manifests as a binary quartic field \cite[(3.3)]{DMT2014}, so one has a {\em Petrov-like classification} based on its (pointwise) root type.  For CR structures, this is the complexification of the degree four part of the Chern--Moser normal equation \cite{ELS1999}.

All multiply-transitive ILC structures in dimension five (in particular having symmetry algebra of dimension $\ge 6$) were classified in~\cite{DMT2014} and organized according to their Petrov type.  In particular, only types III (triple root), D (two double roots), N (quadruple root), or O (flat case) arise.
Non-flat multiply-transitive CR structures necessarily arise as real forms of ILC structures only of types D or N (Corollary \ref{C:CR-DN}).

Any multiply-transitive ILC structure can be encoded by certain algebraic data $(\fs,\fk;\fe,\fv)$, which includes the symmetry algebra $\fs$, two subalgebras $\fe$ and $\fv$ that correspond to $E$ and $V$, and the isotropy subalgebra $\fk = \fe \cap \fv$ of dimension $\geq 1$.  In this paper, we compute CR real forms of this data, which is equivalent to computing anti-involutions $\varphi\colon \fs \to \fs$ that preserve $\fk$ and swap $\fe$ and $\fv$.  (See \S \ref{S:RealForms} and Table \ref{F:AI-reps}.)  Each such real form uniquely defines the local structure of the CR geometry on the homogeneous real hypersurface.  

 The main difficulty is then to find the local equations of real hypersurfaces realizing this algebraic data. To go from an algebraic model to a local realization, we use several techniques.  In \S \ref{S:tubular}, we identify tubular hypersurfaces and, in particular, those that correspond to {\em affine homogeneous} hypersurfaces in $\bbA^3$ (see Tables \ref{F:affineN} and \ref{F:affineD}).  For example, the Winkelmann hypersurface corresponds to the surface in $\bbA^3$ given by the equation $u=xy+x^4$.  In \S \ref{S:Cartan}, we discuss the so-called {\em Cartan hypersurfaces}, which have semisimple symmetry algebra and are treated uniformly in this paper, along with certain related quaternionic models.  Finally, in \S \ref{S:Winkelmann} all remaining local models can be covered by hypersurfaces of \emph{Winkelmann type}, which are given by
 \[
 \Im(w+\bar z_1 z_2) = F(z_1,\bar z_1),
 \]
for some real-valued analytic function $F$.  We can formulate the main result of our paper as follows.

\begin{theorem}
Any multiply-transitive Levi non-degenerate hypersurface in $\bbC^3$ is locally biholomorphically equivalent to one of the following:
\begin{enumerate}
\item the maximally symmetric hypersurfaces $\Im(w) = z_1\bar z_1 \pm z_2\bar z_2$ in $\bbC^3$ with $15$-dimensional symmetry algebra.
\item tubular hypersurfaces listed in Tables \ref{F:affineN} \& \ref{F:affineD} with symmetry algebras of dimension $6$, $7$, or $8$.
\item Cartan hypersurfaces \eqref{E:Cartan-hyp} (see also Table \ref{F:Cartan-hyp}) or the quaternionic models \eqref{E:H-model1}.  These all have symmetry algebra a real form of $\fso(4,\bbC) \cong \fsl(2,\bbC) \times \fsl(2,\bbC)$.
\item hypersurfaces of Winkelmann type given in Table \ref{F:Wink}, having 6-dimensional symmetry algebra:
 \begin{enumerate}
 \item[(i)] $\Im(w+\bar z_1 z_2) = (z_1)^\alpha (\bar{z}_1)^{\bar \alpha}$, where $\alpha\in\bbC \backslash \{ -1, 0, 1, 2\}$;
 \item[(ii)] $\Im(w+\bar z_1 z_2) = \exp(z_1+\bar z_1)$;
 \item[(iii)] $\Im(w+\bar z_1 z_2) = \ln(z_1)\ln(\bar z_1)$.
 \end{enumerate}
 \end{enumerate}
 \end{theorem}

\begin{remark} Among the Winkelmann type hypersurfaces, (ii) and (iii) are equivalent to the tubular hypersurfaces $u = x_1 x_2 + \exp(x_1)$ and $u = x_2\exp(x_1) - \frac{(x_1)^2}{2}$ respectively (see Tables \ref{F:Wink} \& \ref{F:affineN}), while (i) admits a tubular representation if and only if $\frac{(2\alpha-1)^2}{(\alpha+1)(\alpha-2)} \in \bbR$.  (See \S \ref{S:Winkelmann}.)
\end{remark}

Finally in \S \ref{sec:trans}, we prove:

\begin{theorem} \label{T:trans} Up to local biholomorphism, the only locally homogeneous Levi non-degenerate hypersurfaces in $\bbC^3$ whose symmetry algebra is transitive on an open subset of $\bbC^3$ are:
\begin{enumerate}
\item the hyperquadric $\Im(w) = z_1\bar z_1 \pm z_2\bar z_2$;
\item the Winkelmann hypersurface  $\Im(w+\bar z_1 z_2) = |z_1|^4$.
\end{enumerate}
\end{theorem}

 Appendices \ref{A:AI}, \ref{A:tubular}, and \ref{A:Loboda} summarize our classification results for the dimension five case.  Finally, to illustrate our methods in a simpler case, in Appendix \ref{A:CR3} we derive Cartan's classification \cite[bottom of p.70]{Cartan1932a} of (non-flat) homogeneous CR structures in dimension three from the classification of (complex) 2nd order ODE that are homogeneous (in fact, {\em simply}-transitive) under point symmetries.

\section{Complexification of real submanifolds in $\bbC^n$}
\label{S:Complexification}

In this section, we mainly follow~\cite{Merker2008} to establish the relationship between real hypersurfaces in $\bbC^n$ and complete systems of 2nd order PDEs.

\subsection{Complete systems of PDEs defined by real hypersurfaces}\label{S:Complex}

Let $M$ be a real analytic submanifold in $\bbC^{n}$ given by 
 \[
 F_\alpha(z_1,\dots,z_n,\bar z_1, \dots, \bar z_n) = 0,
 \]
where $F_\alpha$ are real analytic functions of the holomorphic and antiholomorphic coordinates. Denote by $\bar\bbC^n$ another copy of $\bbC^n$ with the opposite complex structure, so that the map $\bbC^n\to\bar\bbC^n$ given by $(z_1,\dots,z_n)\mapsto (\bar z_1, \dots, \bar z_n)$ is holomorphic.  Let $(a_1,\dots,a_n)$ be standard coordinates on $\bar\bbC^n$ and define a complex submanifold $M^{c}$ in $\bbC^{n}\times \bar\bbC^{n}$ by:
 \begin{equation}\label{eq:c}
 F_\alpha(z_1,\dots,z_n,a_1,\dots,a_n)=0.
 \end{equation}
 
\begin{defn}
We call $M^{c}$~\emph{the complexification} of a real analytic submanifold in~$\bbC^n$.
\end{defn}

We can regard \eqref{eq:c} as an $n$-parameter family of submanifolds in $\bbC^{n}$.  Let $\cE(M)$ be the corresponding finite-type PDE system whose solution space coincides with this family.
\begin{example}\label{Ex:Complex}
Take the Winkelmann hypersurface $\Im(w+\bar z_1 z_2)=|z_1|^4$.  Its complexification is
\begin{align*}
w - b + a_1 z_2 - z_1 a_2 = 2i(a_1z_1 )^2,
\end{align*}
where $(a_1,a_2,b)$ are holomorphic coordinates on $\bar\bbC^3$.  Regard $w$ as a function of the two independent variables $z_1$ and $z_2$.  Letting $w_j = \frac{\partial w}{\partial z_j}$ and $w_{jk} = \frac{\partial^2 w}{\partial z_j \partial z_k}$, we obtain
 \begin{align*}
 w_1 = a_2 + 4i (a_1)^2 z_1, \quad w_2 = -a_1, \quad w_{11} = 4i (a_1)^2, \quad w_{12} = w_{22} = 0.
 \end{align*}
 Excluding the parameters $(a_1,a_2,b)$, we obtain the PDE system
\[
w_{11} = 4iw_2^2,\quad w_{12}=0,\quad w_{22} = 0.
\]
\end{example}

\begin{prop}\label{p:pde}
Let $M$ be a Levi non-degenerate codimension~1 real analytic submanifold in $\bbC^n$. Then $\cE(M)$ is a complete system of 2nd order PDEs on one function of $(n-1)$ variables.
\end{prop}
\begin{proof}
As shown in~\cite{ChernMoser1974}, locally we can always find a holomorphic coordinate system $(z_1,...,z_{n-1},w)$ such that $M$ is given as:
\[
2\Im(w) = \epsilon_1 \bar z_1 z_1 +\dots + \epsilon_{n-1} \bar z_{n-1} z_{n-1} + F(z_1,\bar z_1, \dots, z_{n-1}, \bar z_{n-1},\bar w),
\]
where $F$ is an analytic function whose Taylor series contains only terms of degree 3 and higher and $\epsilon_j=\pm1$ for $j=1,\dots,n-1$.  The complexification $M^c$ is given by:
\begin{equation}\label{eq:w}
w = b + \epsilon_1 a_1 z_1 + \dots + \epsilon_{n-1} a_{n-1}z_{n-1} + F(z_1,\dots,z_{n-1},a_1,\dots,a_{n-1},b).
\end{equation}
Regarding $w$ as a function of $z_1,\dots,z_{n-1}$, and differentiating \eqref{eq:w} with respect to $z_j$, we get
 \begin{equation}\label{eq:wi}
w_j := \frac{\partial w}{\partial z_j} = \epsilon_j a_j + \frac{\partial F}{\partial z_j}.
\end{equation}

By the implicit function theorem, we can uniquely resolve equations~\eqref{eq:w},~\eqref{eq:wi} in $(a_1,\dots,a_{n-1},b)$ in the neighbourhood of the origin in $\bbC^n$.  Differentiating~\eqref{eq:wi} one more time and substituting there solutions for $(a_1,\dots,a_{n-1},b)$ we obtain the complete system of PDEs of 2nd order.
\end{proof}

It is clear from the construction that if we choose different holomorphic coordinates, this will result in a system of PDEs point equivalent to the initial system. 

\subsection{ILC structures}

Let us recall~\cite{DMT2014} that an \emph{integrable Legendrian contact (ILC) structure} on an odd-dimensional manifold is defined as a contact distribution $C$ decomposed into a sum $C=E\oplus V$ of two completely integrable distributions that are Lagrangian with respect to the (conformal) symplectic form on $C$. 

The above complete system $\cE(M)$ of 2nd order PDEs naturally defines an ILC structure on the space $J^1=J^1(\bbC^{n},n-1)$ of 1-jets of (complex analytic) hypersurfaces in $\bbC^n$. Note that geometrically this space of 1-jets can also be identified with the projectivization of $T^*\bbC^n$. The space $J^1$ carries a natural contact structure $C$. The distribution $V$ is defined as the tangent distribution to the fibers of the projection $J^1\to \bbC^n$. The second complementary integrable
distribution $E\subset C$ is defined by the equation $\cE(M)$ itself. Namely, its fibers are exactly the 1-jets of all its solutions.  (As each solution is uniquely defined by its 1st order derivatives, we see that through each point in $J^1$ goes a unique 1-jet of a solution.)

Suppose the equation $\cE(M)$ is explicitly written as:
\[
\frac{\partial^2 w}{\partial z_j \partial z_k} = f_{jk}(z,w,\partial w),\quad 1\le j,k\le n-1.
\]
On $J^1$, introduce local holomorphic coordinates $(z_j,w,p_j)$, $1\le j\le n-1$, where $p_j=\frac{\partial w}{\partial z_j}$, so that the contact distribution $C$ on $J^1$ is given as:
\[
C = \{ dw - p_1 dz_1 - \dots - p_{n-1}dz_{n-1}  = 0 \}.
\]
The distributions $E$ and $V$ have the form:
\[
 E = \tspan\{ \cD_j := \partial_{z_j} + p_j \partial_w + f_{jk} \partial_{p_k} \}, \qquad V = \tspan\{ \partial_{p_j} \}.
\]
The integrability conditions of $\cE(M)$ ensure that the distribution $E$ is indeed completely integrable.

We can also define a natural ILC structure on $M^c\subset \bbC^n\times \bar\bbC^n$ as follows.  Consider two projections $\bar\pi\colon M^c\to\bar\bbC^n$ and $\pi\colon M^c\to\bbC^n$. They are submersions if and only if the hypersurface $M$ is \emph{holomorphically non-degenerate}, i.e.\ there are no non-zero holomorphic vector fields on $\bbC^n$ tangent to $M$. Define two completely integrable distributions $E$ and $V$ on $M^c$ as tangent distributions to fibers of the projections $\bar\pi$ and $\pi$. Define also $C$ as the sum $E+V$. Similar to Proposition~\ref{p:pde}, it is possible to show that Levi non-degeneracy of $M$ implies that $C$ is a contact distribution on $M^c$.

Instead, we shall show that $M^c$ can be (locally) identified with $J^1$ such that the pairs of distributions $(E,V)$ on $J^1$ and $M^c$  match. Let us assume that $M^c$ is given by:
\[
F(z_1,\dots,z_n,a_1,\dots,a_n)=0.
\] 
Let $(z,a)$ be a point in $M^c$. Consider now a codimension 1 analytic submanifold $S_a\subset \bbC^n$ given by the above equation, where $a\in\bar\bbC^n$ is fixed. Define the map:
\[
\Phi\colon M^c \to J^1, \quad (z,a)\mapsto j^1_z(S_a) = T_zS_a.
\]
Note that $S_a=\pi(\bar\pi^{-1}(a))$, and all such submanifolds are by definition all solutions of $\cE(M)$. This immediately implies that $\Phi$ is a local biholomorphism establishing the equivalence of the pairs of distributions $(E,V)$ on $M^c$ and $J^1$.

 \subsection{Symmetry algebras}

We recall that a holomorphic vector field $X$ on $\bbC^n$ is called an \emph{(infinitesimal) CR symmetry} of the real analytic submanifold $M\subset \bbC^n$ if $X$ is tangent to $M$. This is equivalent to $(X+\bar X)_{M}=0$, or to the fact that the local flow generated by $X$ preserves $M$. The set of all CR symmetries of $M$ forms a real Lie algebra denoted by $\Sym(M)$. We say that $M$ is \emph{(infinitesimally) homogeneous} if $\Sym(M)$ is transitive on $M$, i.e. it spans $TM$ at each point of $M$.

Let $M^{c}\subset \bbC^n\times \bar\bbC^n$ be the complexification of $M$. Denote by $\Sym(M^{c})$ all holomorphic vector fields of the form $X+Y$ tangent to $M^{c}$, where $X$ and $Y$ are holomorphic vector fields on $\bbC^n$ and $\bar\bbC^n$ respectively. It is clear that $\Sym(M^{c})$ is a complex Lie algebra of holomorphic vector fields. We say that $M^{c}$ is \emph{(infinitesimally) homogeneous}, if $\Sym(M^c)$ acts transitively on $M^c$.

\begin{prop}[{\cite[Corollary 6.36]{Merker2008}}]
\label{P:sym}
Assume $M$ is holomorphically non-degenerate (i.e. there is no holomorphic vector field $X\subset TM$). Then $\Sym(M^{c})$ is spanned (as a complex vector space) by vector fields $X+\bar X$, where $X\in \Sym(M)$. Thus, the complex Lie algebra $\Sym(M^{c})$ is a complexification of the real Lie algebra $\Sym(M)$.
\end{prop}

\begin{cor}
The submanifold $M$ is infinitesimally homogeneous if and only if so is the submanifold $M^c$.
\end{cor}

\subsection{Algebraic model of hypersurfaces with transitive symmetry algebra}

Let $M\subset \bbC^n$ be a Levi non-degenerate hypersurface with transitive symmetry algebra $\Sym(M)$. Consider its complexification $M^{c}$ and let $\fs = \Sym(M^{c})$. By above, $\fs$ is infinitesimally transitive on $M^{c}$. 

Let $z^0$ be an arbitrary point of $M\subset \bbC^n$. Then by definition the point $(z^0,\bar z^0) \subset \bbC^n\times \bar\bbC^n$ lies in $M^c$.  Let $\fk$ be a subalgebra of $\fs$ consisting of all vector fields that vanish at $(z^0,\bar z^0)$. Since $\fs$ is transitive, $\fk$ has codimension $2n-1$ in $\fs$. 

As above, let $E$ and $V$ be two completely integrable distributions on $M^c$ defining an ILC structure on it. Denote by $\fe$ and $\fv$ two subspaces in $\fs$ consisting of such vector fields $X$ that $X_{(z^0,\bar z^0)}\in E_{(z^0,\bar z^0)}$ and $X_{(z^0,\bar z^0)}\in V_{(z^0,\bar z^0)}$ respectively. Since $E$ and $V$ are completely integrable, it follows that both $\fe$ and $\fv$ are actually subalgebras in $\fs$. It is clear that $\fe\cap \fv = \fk$ and that $\fe+\fv$ is a subspace of codimension 1 in~$\fs$.

The fact that $E+V$ is a contact structure on $M^c$ can be translated to the algebraic language as follows. Consider the bilinear map:
\[
\fe/\fk \times \fv/\fk \to \fs/(\fe+\fv),\quad (X+\fk,Y+\fk)\mapsto [X,Y] + (\fe+\fv).
\]  
It is easy to see that it is well-defined and is non-degenerate.

We call the tuple $(\fs,\fk;\fe,\fv)$ \emph{an algebraic model of the ILC structure $(E,V)$ on $M^c$}. It uniquely determines the local ILC structure on $M^c$ in a neighbourhood of the point $(z^0,\bar z^0)$.

Consider now the involutive map:
\[
\bbC^n\times \bar\bbC^n \to \bbC^n\times \bar\bbC^n, \quad (z,a)\mapsto (\bar a, \bar z).
\]
By definition, it stabilizes $M^c$ and preserves $\fs = \Sym(M^c)$. Its restriction to $\fs$ defines an anti-involution $\varphi$ of $\fs$ that preserves $\fk$ and swaps $\fe$ and $\fk$. The tuple $(\fs,\fk;\fe,\fv)$ with the anti-involution $\varphi$ uniquely determines the local structure of $M$ itself in the neighbourhood of the point~$z^0$.

\section{Classification of real forms}
\label{S:RealForms}

 Let $(\fs,\fk;\fe,\fv)$ be the algebraic data associated to a locally homogeneous complex ILC structure.  This satisfies the following properties:
 \begin{itemize}
 \item $\fk \subset \fe, \fv \subset \fs$ are Lie subalgebra inclusions, and $\fe \cap \fv = \fk$;
 \item $\fe + \fv$ has codimension one in $\fs$, and $[\fe,\fv] \not\subset \fe + \fv$.
 \end{itemize}
 Recall that any real form of $\fs$ is the fixed point set $\fs^\varphi$ of an anti-involution $\varphi : \fs \to \fs$, i.e.\ a complex anti-linear map satisfying $\varphi^2 = \id$ and $\varphi([x,y]) = [\varphi(x),\varphi(y)]$ for any $x,y \in \fs$.  We say that $\varphi$ is {\em admissible} if: (i) it preserves $\fk$, and (ii) it swaps $\fe$ and $\fv$.  Any homogeneous CR structure is obtained from some admissible anti-involution $\varphi$ for a homogeneous complex ILC structure.  Indeed, from $(\fs^\varphi,(\fe + \fv)^\varphi,\fk^\varphi)$, the contact subspace $C$ corresponds to $(\fe + \fv)^\varphi \mod \fk^\varphi$, and we designate $E$ and $V$ to be the $+i$-eigenspace $C^{1,0}$ and $-i$-eigenspace $C^{0,1}$ under (the $\bbC$-linear extension of) $J$.  The Levi form $[\xi,\bar\eta]\,\mod C^\bbC$, for $\xi,\eta \in \Gamma(C^{0,1})$, can then be evaluated from the above Lie algebraic data.
  
 We say that two admissible anti-involutions $\varphi,\psi$ are {\em equivalent} if $\psi = T \circ \varphi \circ T^{-1}$, where $T$ is an {\em admissible} automorphism of $\fs$, i.e.\ (i) preserves $\fk$, and (ii) swaps $\fe,\fv$ or preserves both of them.  Equivalent admissible anti-involutions yield isomorphic homogeneous CR structures, so it suffices to identify representatives from each equivalence class.
 
 \begin{theorem} \label{T:AI} A complete list of representative admissible anti-involutions for all non-flat 5-dimensional multiply-transitive complex ILC structures is given in Table \ref{F:AI-reps}.
 \end{theorem}
 
The proof of Theorem \ref{T:AI} is a straightforward, but tedious, computation.  We will outline the details for some examples. 

 All non-flat 5-dimensional multiply-transitive complex ILC structures were classified in \cite[Tables 4.2-4.4]{DMT2014}.  The structure equations for any model given there are written with respect to an adapted basis $\varpi_1,...,\varpi_5,...$, and we let $e_1,...,e_5,...$ denote the dual basis.  In this {\em Cartan basis}, $\fk = \tspan\{ e_6,... \}$, $\fe = \tspan\{ e_1, e_2 \} + \fk$, and $\fv = \tspan\{ e_3, e_4 \} + \fk$.
 
 Note that in swapping $\fe$ and $\fv$, any admissible anti-involution $\varphi$ must swap their respective derived series.  Thus, no admissible anti-involutions exist for:
 \begin{itemize}
 \item N.7-1: $\dim(\fe^{(1)}) = 2$ and $\dim(\fv^{(1)}) = 3$;
 \item D.6-4: $\dim(\fe^{(1)}) = 3$ and $\dim(\fv^{(1)}) = 2$;
 \item III.6-2: $\dim(\fe^{(1)}) = 2$ and $\dim(\fv^{(1)}) = 1$.
 \end{itemize}
 
 For the next three examples, we refer to the structure equations in Table \ref{F:streq}.  These have $\dim(\fs) = 6$, so $\fe = \tspan\{ e_1, e_2, e_6 \}$, $\fv = \tspan\{ e_3, e_4, e_6 \}$, $\fk = \tspan\{ e_6 \}$, and so $\varphi(e_6) = \lambda e_6$ with $|\lambda| = 1$.  Let
 \[
 \sigma_{jk} = [\varphi(e_j), \varphi(e_k)] - \varphi([e_j,e_k]).
 \]
 
 \begin{table}[h]
 \begin{framed}
 {\bf III.6-1}: 
  \[
 \begin{array}{c|cccccc}
  & e_1 & e_2 & e_3 & e_4 & e_5 & e_6\\ \hline
 e_1 & \cdot & \frac{5}{4} e_2 & \frac{3}{2} e_3 - e_5 & -e_1 + \frac{1}{2} e_4 - \frac{9}{8} e_6 & \frac{1}{2} e_2 - \frac{3}{16} e_3 + e_5 & \cdot\\
 e_2 && \cdot & \cdot & e_2 + \frac{3}{4} e_3 - e_5 & \cdot & 2 e_2\\
 e_3 &&& \cdot & 3 e_3 & \cdot & 2 e_3\\
 e_4 &&&& \cdot & -\frac{3}{4} e_3 - 2 e_5 & \cdot\\
 e_5 &&&&& \cdot & 2 e_5\\
 e_6 &&&&&&\cdot \\
 \end{array}
 \]
 {\bf D.6-1}: 
 \[
 \begin{array}{c|cccccc}
  & e_1 & e_2 & e_3 & e_4 & e_5 & e_6\\ \hline
 e_1 & \cdot & \cdot & -e_5 + \frac{1}{4} e_6 & -\sqrt{2} e_2 & 3 e_1 & -4 e_1\\
 e_2 && \cdot & \sqrt{2} e_4 & -e_5 - \frac{3}{4} e_6 & \frac{3}{2} e_2 & -2 e_2\\
 e_3 &&& \cdot & \cdot & -3 e_3 & 4 e_3\\
 e_4 &&&& \cdot & -\frac{3}{2} e_4 & 2 e_4\\
 e_5 &&&&& \cdot & \cdot\\
 e_6 &&&&&&\cdot \\
 \end{array}
 \]
 {\bf N.6-2}:
 \[
 \begin{array}{c|cccccc}
  & e_1 & e_2 & e_3 & e_4 & e_5 & e_6\\ \hline
 e_1 & \cdot & -2 a e_2 - e_6 & -a e_3 - e_5 & -e_6 & -e_3 - 2 a e_5 & -e_2 - a e_6\\
 e_2 && \cdot & \cdot & b e_2 - e_5 & \cdot & \cdot\\
 e_3 &&& \cdot & 2 b e_3 - e_6 & \cdot & \cdot\\
 e_4 &&&& \cdot & e_2 - 2b e_5 & e_3 - b e_6\\
 e_5 &&&&& \cdot & \cdot\\
 e_6 &&&&&&\cdot \\
 \end{array}
 \]
 \end{framed}
 \label{F:streq}
 \caption{Some ILC structures in the Cartan basis}
 \end{table}
 
 {\bf III.6-1}:  Since $\fe^{(1)} = \tspan\{ e_2 \}$ and $\fv^{(1)} = \tspan\{ e_3 \}$, then $\varphi(e_2) = s e_3$ and $\varphi(e_3) = \frac{1}{\bar{s}} e_2$.  The maximal abelian subalgebras of $\fe$ and $\fv$ containing $\fk$ must be swapped, so $\varphi(e_1) = \alpha e_4 + \beta e_6$ and $\varphi(e_4) = \gamma e_1 + \delta e_6$, with $\gamma \alpha \neq 0$.  Now $\sigma_{26} = 0$ forces $\lambda = 1$, and $\sigma_{14} = 0$ forces $\gamma = 2$ and $\alpha = \frac{1}{2}$, and $\beta = \frac{\delta}{2} - \frac{9}{4}$.  But $\sigma_{12} = \sigma_{34} = 0$ implies $\beta = -\frac{11}{8}$ and $\delta = \frac{11}{4}$, which contradicts the previous relation.  Thus, there are no CR structures associated with the type III models in \cite{DMT2014}.
 
 \begin{cor} \label{C:CR-DN}
 In dimension five, all non-flat multiply-transitive Levi-non-degenerate CR structures complexify to multiply-transitive complex ILC structures of type D or N.
 \end{cor}
 
 {\bf D.6-1}: We must swap $\fe^{(1)} = \tspan\{ e_1, e_2 \}$ and $\fv^{(1)} = \tspan\{ e_3, e_4 \}$.  Now $\sigma_{16} = \sigma_{26} = 0$ implies that $\ad(e_6)|_{\fv^{(1)}} = \diag(\frac{4}{\lambda}, \frac{2}{\lambda})$ in the basis $\{ \varphi(e_1),\varphi(e_2)\}$.  But also $\ad(e_6)|_{\fv^{(1)}} = \diag(-4,-2)$ in the basis $\{ e_3, e_4 \}$.  Thus, $\lambda = -1$ and $\varphi(e_1) = s e_3$, $\varphi(e_2) = t e_4$.  Since $\varphi^2 = \id$, then $\varphi(e_3) = \frac{1}{\bar{s}} e_1$ and $\varphi(e_4) = \frac{1}{\bar{t}} e_2$.  Now $\sigma_{14} = 0$ leads to $s = |t|^2 \in \bbR \backslash \{ 0 \}$.  The admissible automorphism $(e_1,...,e_6) \mapsto (\frac{e_1}{c^2}, \frac{e_2}{c}, c^2 e_3, c e_4, e_5, e_6)$ induces $s \mapsto \frac{s}{|c|^4}$, so we may normalize $s = 1$ and $t = \epsilon = \pm 1$.  Finally, $\sigma_{13} = 0$ implies $\varphi(e_5) = -e_5$.  Thus, we have two real forms.
 The fixed point Lie algebra $\fs^\varphi$ has (real) basis
 \begin{align*}
 E_1 &= e_1 + e_3, \quad E_2 = i(e_1 - e_3), \quad 
 E_3 = e_2 + \epsilon e_4, \quad E_4 = i(e_2 - \epsilon e_4), \\ 
 E_5 &= i e_5, \quad E_6 = i e_6.
 \end{align*}
 The contact subspace $C$ is spanned by $E_1,..., E_4 \mod \fk$.  In this basis, $J = \begin{psmallmatrix} 
 0 & -1 & 0 & 0\\
 1 & 0 & 0 & 0\\
 0 & 0 & 0 & -1\\
 0 & 0 & 1 & 0
 \end{psmallmatrix}$, with $C^{0,1}$ associated to $\fv/\fk = \tspan\{ e_3, e_4 \} \mod \fk$.  Since $\varphi$ sends $(e_3,e_4) \mapsto (e_1,\epsilon e_2)$, then the Levi form is represented by
 $\begin{psmallmatrix}
 [e_3,e_1] & [e_3, \epsilon e_2]\\
 [e_4,e_1] & [e_4, \epsilon e_2]
 \end{psmallmatrix} \equiv 
 \begin{psmallmatrix}
 1 & 0\\
 0 & \epsilon
 \end{psmallmatrix} e_5 \mod (\fe + \fv)$.  This is definite if and only if $\epsilon = 1$.\\

 {\bf N.6-2}: This family is parametrized by $(a,b) \in \bbC^2$, with the redundancy that $(a,b)$, $(-a,b)$, $(a,-b)$, $(-a,-b)$ all yield equivalent models.\footnote{The map $(e_1,...,e_6) \mapsto (\epsilon_1 e_1, \epsilon_2 e_2, \epsilon_1 e_3, \epsilon_2 e_4, e_5, \epsilon_1 \epsilon_2 e_6)$, where $\epsilon_1 = \pm 1$ and $\epsilon_2 = \pm 1$, induces the parameter change $(a,b) \mapsto (\epsilon_1 a, \epsilon_2 b)$ and $\epsilon \mapsto \epsilon_1 \epsilon_2 \epsilon$.  This is an automorphism only when $a=b=0$.}  Thus, we can consider $(a^2,b^2) \in \bbC^2$ as the essential parameters.  Each of $\fe$ and $\fv$ contains a unique maximal abelian subalgebra containing $\fk$, namely $\tspan\{ e_2, e_6 \}$ and $\tspan\{ e_3, e_6 \}$ respectively.  These must be swapped by $\varphi$.  For $s_1 t_1 \neq 0$,
 \begin{align*}
 \varphi(e_1) = s_1 e_4 + s_2 e_3 + s_3 e_6, \quad \varphi(e_2) = t_1 e_3 + t_2 e_6.
 \end{align*}
 Then $\sigma_{16} = 0$ yields \framebox{$t_1 = -s_1 \lambda$} and $t_2 + bt_1 + \bar{a} \lambda = 0$.  From $\sigma_{12} = 0$,
 \begin{itemize} 
 \item $0 = s_1 t_2 - 2s_1 t_1 b + 2\bar{a} t_1 = s_1 t_2 - 2s_1 (-t_2 - \bar{a}\lambda) + 2\bar{a} t_1 = 3 s_1 t_2$, so \framebox{$t_2 = 0$}.  It follows that $\bar{a} = bs_1$.
 \item $0 = s_1 t_1 + \lambda = \lambda(1 -(s_1)^2)$, so \framebox{$s_1 = \epsilon = \pm 1$}, hence \framebox{$\bar{a} = \epsilon b$}.
 \end{itemize}
 
 We have $e_2 = \varphi^2(e_2) = \varphi(t_1 e_3) = -\epsilon \bar\lambda \varphi(e_3)$.  Similarly, $e_1 = \varphi^2(e_1)$ yields $\varphi(e_4) = \epsilon e_1 + \overline{(s_2/\lambda)}  e_2 - \epsilon \overline{s_3} \lambda e_6$.  Now we obtain
 \begin{align*}
 0 &= \sigma_{14} = [\varphi(e_1),\varphi(e_4)] - \varphi([e_1,e_4])\\
 &= \epsilon (s_3 - b \overline{(s_2/\lambda)}) e_2 + \epsilon(- \epsilon \overline{s_3} \lambda + s_2 a) e_3 + \epsilon (s_2 + \overline{(s_2/\lambda)}) e_5 + (\epsilon s_3 a + \overline{s_3} \lambda b + 1 + \lambda) e_6
 \end{align*}
 The coefficients of $e_2,e_5$ imply \framebox{$s_3 = -s_2 b$}.  The rest reduce to $\bar\lambda s_2 + \bar{c}_2 = 0 = (1+\lambda)(1 - s_2|b|^2)$.  We must have \framebox{$\lambda = -1$}.  (Note $s_2 = \frac{1}{|b|^2}$ also forces $\lambda = -1$.)  Thus, \framebox{$s_2 \in \bbR$} and \framebox{$\varphi(e_6) = -e_6$}.
 
 For any $r$, the linear map $T$ fixing $e_2,e_3,e_5,e_6$ and sending
 \[
 e_1 \mapsto e_1 - r e_2 - ar e_6, \quad e_4 \mapsto e_4 + re_3 - rbe_6
 \]
 is an automorphism of $\fs$ that preserves each of $\fk, \fe, \fv$.  In the new basis $\tilde{e}_1 = T(e_1)$, ..., $\tilde{e}_6 = T(e_6)$, we have $\varphi(\tilde{e}_1) =  \epsilon \tilde{e}_4 + \tilde{c}_2 \tilde{e}_3 + \tilde{c}_3 \tilde{e}_6$, where $\tilde{c}_3 = -\tilde{c}_2 b$ and (using $t_1 = -s_1\lambda = \epsilon$) we have $\tilde{c}_2 = s_2 - r \epsilon - \bar{r} t_1 = s_2 - (r + \bar{r}) \epsilon$.
 Since $s_2 \in \bbR$, we may normalize \framebox{$s_2 = 0$} (and hence \framebox{$s_3 = 0$}).  Hence, $\varphi$ maps $(e_1,e_2,e_3,e_4) \mapsto (\epsilon e_4, \epsilon e_3, \epsilon e_2, \epsilon e_1)$, and
 \begin{align*}
 0 &= \sigma_{13} = [\varphi(e_1),\varphi(e_3)] - \varphi([e_1,e_3]) = [e_4,e_2] + \varphi(ae_3 + e_5) \\
 &= e_5 - be_2 + \epsilon \bar{a} e_2 + \varphi(e_5) = e_5 + \varphi(e_5) \qRa \framebox{$\varphi(e_5) = -e_5$}.
 \end{align*}
 Since $\varphi$ sends $(e_3,e_4) \mapsto (\epsilon e_2, \epsilon e_1)$, then 
 $\begin{psmallmatrix}
 [e_3, \epsilon e_2] & [e_3, \epsilon e_1]\\ 
 [e_4, \epsilon e_2] & [e_4, \epsilon e_1]
 \end{psmallmatrix} \equiv
 \begin{psmallmatrix} 
 0 & \epsilon \\ 
 \epsilon & 0 
 \end{psmallmatrix} e_5$ implies an indefinite Levi-form.
 
 We obtained a unique representative admissible anti-involution:
 \begin{itemize}
 \item $ab \neq 0$: Since $\bar{a} = \epsilon b$, then $\epsilon$ is uniquely determined.
 \item $a=b=0$: Rescaling $(e_2,e_4,e_6)$ by $\epsilon$ normalizes $\epsilon = 1$.
 \end{itemize} 
 Using the aforementioned parameter redundancy, we normalize $\epsilon = 1$ (so that $b = \bar{a}$).  Thus, $b^2 = \bar{a}^2 \in \bbC$ are the parameters yielding CR structures, and in each case there is a {\em unique} structure.

 \begin{remark} \label{R:N62} For N.6-2, the duality swap induces $(a,b) \mapsto (b,a)$ (see \cite[Table A.6]{DMT2014}), so the structure is self-dual if and only if $b^2=a^2$.  For the cases admitting CR structures, $b^2 = \bar{a}^2$, so these are self-dual precisely when $b^2 = a^2 \in \bbR$.  As shown in \S \ref{S:tubular}, these coincide with the cases that admit tubular representations -- see Table \ref{F:affineN} for the tubular models.
 \end{remark}
 
All admissible anti-involutions can be computed in the same way.  The final  list is presented in Table~\ref{F:AI-reps}.  The local models for all these anti-involutions are constructed in the following sections.

\section{Homogeneous tubular hypersurfaces}
\label{S:tubular}

 A natural class of CR structures are {\em tubular hypersurfaces}, which arise from analytic hypersurfaces in $\bbR^n$ (i.e.\ their ``base'').  In $\bbC^3$, the majority of the hypersurfaces in our classification are indeed tubular (Theorem \ref{s4:t1}).  A complete classification of {\em affine-homogeneous} surfaces in $\bbR^3$ was obtained by Doubrov--Komrakov--Rabinovich \cite[Theorem 1]{DKR1995}, so using their list is a natural starting point for our study.  However, not all (CR-)homogeneous tubular hypersurfaces in $\bbC^3$ have affine-homogeneous base, so it is important to be able to abstractly identify tubular CR structures and determine the affine symmetry dimension for their base hypersurfaces.

Consider an analytic hypersurface in $\bbR^n$:
\begin{equation}
\label{s4:aff_eq}
f(x_1,\dots,x_n) = 0.
\end{equation} 
 A \emph{tubular hypersurface} $M$ in $\bbC^n$ induced by \eqref{s4:aff_eq} is defined by the equation
\begin{equation} 
\label{s4:tube}
 f(\Re(z_1),\dots,\Re(z_{n})) = 0.
\end{equation}
Obviously, this hypersurface admits the symmetries $i\partial_{z_1}, ..., i\partial_{z_{n}}$. Now \eqref{s4:tube} can be rewritten as
\begin{equation}\label{s4:real_eq} f\left(\frac{z_1+\bar z_1}2,\dots,\frac{z_{n}+\bar z_{n}}2\right)=0.
\end{equation}
From \S \ref{S:Complexification}, the complexification $M^c$ of \eqref{s4:real_eq} is the following complex submanifold of $\bbC^{n}\times\bar\bbC^{n}$:
\begin{equation}\label{s4:surf_f}
 f\left(\frac{z_1+a_1}2,\dots,\frac{z_{n}+a_{n}}2\right)=0. 
\end{equation}
\begin{defn}
We call a complex ILC structure given by \eqref{s4:surf_f} a \emph{tubular ILC structure}.
\end{defn}
Equation \eqref{s4:surf_f} can be seen as a (translation-invariant) family of hypersurfaces in $\bbC^n$ parametrised by $a = (a_1,...,a_n)$. 
The real hypersurface \eqref{s4:real_eq} is the fixed point set of the anti-involution
\begin{equation}\label{s4:anti-inv}
\tau\colon \bbC^{n}\times\bar\bbC^n\to \bbC^{n}\times\bar\bbC^n,\quad\tau(z,a)=(\bar a, \bar z). 
\end{equation}

If the real hypersurface \eqref{s4:aff_eq} admits affine symmetries, then these symmetries 
can be extended to the complex-affine symmetries of \eqref{s4:tube} in $\bbC^n$. More precisely, if $\phi\colon\bbR^n\to\bbR^n$ is the affine symmetry 
\[ \phi(x)=Ax+B,\quad A\in\GL(n,\bbR),\,\, B\in\bbR^n, \]
 then for $z = x+iy$, the transformation $z \mapsto Az+B$ is the symmetry of the corresponding real hypersurface in $\bbC^n$. The complex-affine symmetries form a subalgebra of the CR symmetry algebra.

 Recall \cite{DMT2014} that given an ILC structure with integrable subbundles $E$ and $V$, the dual ILC structure is obtained by swapping $E$ and $V$.

\begin{prop}\label{s4:prop_SD}
A tubular ILC structure is self-dual.
\end{prop}
\begin{proof}
The involution $\sigma(z,a)=(a, z)$ preserves \eqref{s4:surf_f} and swaps variables $z_j$ with parameters $a_j$. This means that $\sigma$ is a duality transformation for the ILC structure $M^c$.
\end{proof}

It is well known that a hypersurface with non-degenerate second fundamental form induces a hypersurface in $\bbC^n$ with non-degenerate Levi form of the same signature. To see this, consider an analytic hypersurface in $\bbR^n$. Using affine transformations, we can assume it is of the form:
\[ 
u=g(x_1,\dots,x_{n-1})=\epsilon_1x_1^2+\dots+\epsilon_{n-1} x_{n-1}^2 + O(|x|^3), \qquad \epsilon_j = \pm 1,
\]
where $u = x_n$. The corresponding tubular hypersurface in $\bbC^n$ is:
\[
\frac{w+\bar w}2 =g\left(\frac{z_1+\bar z_1}2,\dots,\frac{z_{n-1}+\bar z_{n-1}}2\right)=\epsilon_1\left(\frac{z_1+\bar z_1}2\right)^2+\dots+\epsilon_{n-1} \left(\frac{z_{n-1}+\bar z_{n-1}}2\right)^2  + O(|z|^3).
\]
The holomorphic coordinate change $w \mapsto w + \frac{1}{2} (\epsilon_1 z_1^2 + \dots + \epsilon_{n-1} z_{n-1}^2)$ transforms it to:
\[ 
\Re(w)=\epsilon_1 |z_1|^2 + \dots + \epsilon_{n-1} |z_{n-1}|^2 + O(|z|^3).
\]

Henceforth assuming non-degeneracy, take $M \subset \bbR^n$ of the form $u = g(x_1,...,x_{n-1})$ and with $g$ having nonzero Hessian.  Letting $w = z_n$, we see that $M^c$ is of the form
\begin{equation*}
\frac{w+c}2= g\left(\frac{z_1+a_1}2,\dots,\frac{z_{n-1}+a_{n-1}}2\right). 
\end{equation*}
Differentiate this twice with respect to $z_j$ to obtain
 \[ 
w_j = g_j\left(\frac{z_1+a_1}2,\dots,\frac{z_{n-1}+a_{n-1}}2\right), \quad w_{jk}=\frac{1}{2} g_{jk}\left(\frac{z_1+a_1}2,\dots,\frac{z_{n-1}+a_{n-1}}2\right).
 \]
 Since $\operatorname{Hess}(g)\neq 0$, the first set of equations can be locally solved for $z_1+a_1,...,z_{n-1} + a_{n-1}$.  Substitution into the second set of equations yields the 2nd order PDE system
\begin{equation}\label{s4:tub_ilc} w_{jk}=G_{jk}\left(w_1,\dots,w_{n-1}\right).
\end{equation}
 The translation group acts locally transitively on the space of solutions. Since hypersurfaces with non-degenerate 2nd fundamental form cannot admin one-parameter groups of translations, the infinitesimal stabilizer should be trivial at each point in the solution space (a hypersurface in $\bbC^n$).

Recall from Section \ref{S:Complexification} that a complex ILC on $M^c$ can be regarded as a double fibration over the base manifold $M^c/V=\bbC^n$ and the solution space $M^c/E=\bar\bbC^n$.
\begin{lemma}\label{L:proj}
Any non-zero symmetry of the complex ILC structure on the manifold $M^c$ has non-zero projections on $\bbC^n$ and $\bar\bbC^n$.
\end{lemma}
\begin{proof}
Without loss of generality assume that $X$ is a symmetry of the ILC structure on $M^c$ which projects  trivially on $\bar\bbC^n$. This implies $X\in \Gamma(E)$. From the definition of ILC structure it follows that for every point $p\in M^c$ there exists $Y\in \Gamma(V)$, such that $[X,Y]_p\not \in E_p\oplus V_p$. But then the field $X$ does not preserve $V$.
\end{proof}

Every symmetry of PDE \eqref{s4:tub_ilc} induces an action on the solution space. Therefore for every symmetry $X\in\fX( \bbC^n)$ of \eqref{s4:tub_ilc}, Lemma \ref{L:proj} gives a unique $Y\in\fX(\bar\bbC^n)$ such that $X+Y$ is tangent to \eqref{s4:surf_f}. We call $X+Y$ \emph{the prolongation of the symmetry $X$ to the solution space}.

\begin{example}\label{Ex:D7} Consider the following affine surface in $\bbR^3$:
 \begin{align} \label{E:homAffex}
 u = \alpha\ln(x) + \ln(y), \quad \alpha \in \bbR \setminus \{ 0, -1 \}.
 \end{align}
It is affinely homogeneous~\cite{DKR1995} and gives rise to the tubular hypersurface $M \subset \bbC^3$ given by
 \begin{equation}\label{s4:CRex}
\Re(w) = \alpha\ln(\Re(z_1)) + \ln(\Re(z_2)).
 \end{equation}
 Complexifying \eqref{s4:CRex}, we get a 3-parameter family of surfaces in $\bbC^3$, parametrized by $(a_1,a_2,b) \in \bbC^3$:
 \begin{align} \label{s4:D7-soln}
 w = 2\alpha\ln\left(\frac{z_1+a_1}2\right) + 2 \ln\left(\frac{z_2+a_2}2\right) - b.
 \end{align}
 Differentiating \eqref{s4:D7-soln} twice yields
 \[
 w_1 = \frac{2\alpha}{z_1+a_1}, \quad
 w_2 = \frac{2}{z_2+a_2}, \quad w_{11} = -\frac{2\alpha}{(z_1+a_1)^2}, \quad w_{12} = 0, \quad w_{22} = -\frac{2}{(z_2+a_2)^2}.
 \]
 We eliminate $(a_1,a_2,b)$ in $w_{11},w_{12},w_{22}$ using the equations for $w,w_1,w_2$, and obtain
 \begin{align} \label{s4:D7}
 w_{11} = -\frac{(w_1)^2}{2\alpha}, \quad w_{12} = 0, \quad w_{22} = -\frac{(w_2)^2}{2}.
 \end{align}
 
Using \cite[(3.3)]{DMT2014}, this complex ILC structure has type D harmonic curvature.  More precisely, from \cite[Table 1.1]{DMT2014}, it is point-equivalent to the D.7 model $w_{11} = (w_1)^2, w_{12} = 0, w_{22} = \lambda (w_2)^2$ with $\lambda=\alpha$.  An abstract description for D.7 \cite[Table 4.3]{DMT2014} is given in terms of a parameter $a$, and \cite[Table A.4]{DMT2014} gives $\lambda = \frac{3+4a}{3-4a}$, hence $a = \frac{3}{4} (\frac{\alpha-1}{\alpha+1}) \in \bbR \backslash \{ \pm \frac{3}{4} \}$.  The redundancy $a \mapsto -a$ induces $\alpha \mapsto \frac{1}{\alpha}$.

 The point symmetries of \eqref{s4:D7} can be easily computed, for example in {\tt Maple} via:
 \begin{verbatim}
 with(DifferentialGeometry): with(GroupActions):
 DGsetup([z1,z2,w,w1,w2],M):
 F:=-w1^2/2/alpha: G:=0: H:=-w2^2/2:
 E:=evalDG([D_z1+w1*D_w+F*D_w1+G*D_w2,D_z2+w2*D_w+G*D_w1+H*D_w2]):
 V:=evalDG([D_w1,D_w2]):
 InfinitesimalSymmetriesOfGeometricObjectFields([E,V],output="list");
 \end{verbatim}
 This yields holomorphic vector fields on the jet space $J^1(\bbC^3,2)$ that are projectable over $\bbC^3$.  On the latter space, these are given by
 \begin{equation}\label{s4:ilc_sym}
 \partial_{z_1}, \quad \partial_{z_2}, \quad \partial_w, \quad z_1\partial_{z_1}, \quad z_2\partial_{z_2}, \quad 
 (z_1)^2 \partial_{z_1} + 2\alpha z_1\partial_w, \quad (z_2)^2 \partial_{z_2} + 2 z_2 \partial_w.
 \end{equation}
  
Consider the vector fields
 \begin{equation}\label{s4:ilc_dual_sym}
 \partial_{a_1}, \quad \partial_{a_2}, \quad \partial_b, \quad {a_1}\partial_{a_1}, \quad b\partial_{a_2}, \quad 
  (a_1)^2 \partial_{a_1} + 2\alpha{a_1}\partial_b, \quad b^2 \partial_{a_2} + 2 b \partial_b,
 \end{equation}
which are obtained by replacing $z_1$, $z_2$, $w$ in \eqref{s4:ilc_sym} with $a_1$, $a_2$, $b$. The vector fields \eqref{s4:ilc_dual_sym} are projections of ILC symmetries on $(a_1,a_2,b)$-space due to self-duality of tubular ILC structures. By Lemma \ref{L:proj}, for every vector field $X$ in the linear span of \eqref{s4:ilc_sym}, there exists a unique vector field $Y$ in the linear span of \eqref{s4:ilc_dual_sym} such that $X+Y$ is tangent to \eqref{s4:D7-soln}, i.e. $X+Y$ is the prolongation of $X$ to the solution space. 
Here, the prolonged symmetry algebra is spanned by:
 \begin{align*}
 & \partial_{z_1} - \partial_{a_1}, \quad \partial_{z_2} - \partial_{a_2}, \quad \partial_w - \partial_b, \quad
  x\partial_{z_1} +\alpha\partial_w + a_1\partial_{a_1} + \alpha \partial_b, \quad a_2\partial_{z_2}+   \partial_w + a_2\partial_{a_2} + \partial_b,\\
 & \label{s4:sym3} (z_1)^2 \partial_{z_1} + 2 \alpha z_1\partial_w - (a_1)^2 \partial_{a_1} - 2 \alpha a_1 \partial_b, \quad
 (z_2)^2 \partial_{z_2} + 2 z_2\partial_w - (a_2)^2 \partial_{a_2} - 2 a_2 \partial_b.
\end{align*}
 The $\tau$-stable subspace (see \eqref{s4:anti-inv} for $\tau$) immediately gives the CR symmetry of \eqref{s4:CRex}:
\begin{align*}
& i\partial_{z_1}, \quad i\partial_{z_2}, \quad i\partial_w, \quad
z_1\partial_{z_1} + \alpha\partial_w, \quad z_2\partial_{z_2} +  \partial_w, \\
& i(z_1)^2 \partial_{z_1} + 2i \alpha z_1 \partial_w, \quad
i(z_2)^2 \partial_{z_2} + 2 iz_2 \partial_w,
\end{align*}
 and this is isomorphic to $\fsl(2,\bbR) \times \fsl(2,\bbR) \times \bbR$.
(We use the common convention of suppressing the explicit action on $\bar z_j$.)  Note that $z_1\partial_{z_1} + \alpha\partial_w$ and $z_2\partial_{z_2} + \partial_w$ are affine symmetries of \eqref{s4:CRex}.

We already know that the complex ILC structure corresponding to this model is D.7 with $a = \frac{3}{4} (\frac{\alpha-1}{\alpha+1}) \in \bbR \backslash \{ \pm \frac{3}{4} \}$.  (Recall that the essential parameter is $a^2$ here.)  To complete the abstract classification of these models, we must determine the anti-involution in Table \ref{F:AI-reps}.  First note that \eqref{E:homAffex} has Hessian matrix $\begin{psmallmatrix} u_{xx} & u_{xy} \\ u_{xy} & u_{yy} \end{psmallmatrix} = \diag( -\frac{\alpha}{x^2}, -\frac{1}{y^2})$, so the Levi-form of \eqref{s4:CRex} has definite signature iff $\alpha > 0$ iff $|a| < \frac{3}{4}$.  From the abstract D.7 structure equations in \cite[Table 4.3]{DMT2014}, we can identify the real form $\fs^\varphi$ of $\fs = \fsl(2,\bbC) \times \fsl(2,\bbC) \times \bbC$ for each anti-involution $\varphi$ in Table \ref{F:AI-reps} by examining the signature of the Killing form for the semisimple part of $\fs^\varphi$.  Furthermore, the parameter redundancy $a \mapsto -a$ induces $\varphi_1^{(\epsilon_1,\epsilon_2)} \mapsto \varphi_1^{(\epsilon_2,\epsilon_1)}$, so we obtain:
\[
\begin{array}{|c|c|c|c|c} \hline
\mbox{Anti-involution} & |a| < \frac{3}{4} & a > \frac{3}{4} & \mbox{Levi-form signature}\\ \hline\hline
 \varphi_1^{(1,1)} & \fsu(2) \times \fsu(2) \times \bbR & \fsl(2,\bbR) \times \fsu(2) \times \bbR  & \mbox{definite}\\
 \varphi_1^{(1,-1)} & \fsl(2,\bbR) \times \fsu(2) \times \bbR & \fsl(2,\bbR) \times \fsl(2,\bbR) \times \bbR & \mbox{indefinite} \\
 \varphi_1^{(-1,1)} & \fsl(2,\bbR) \times \fsu(2) \times \bbR & \fsu(2) \times \fsu(2) \times \bbR & \mbox{indefinite}\\
 \varphi_1^{(-1,-1)} & \fsl(2,\bbR) \times \fsl(2,\bbR) \times \bbR & \fsl(2,\bbR) \times \fsu(2) \times \bbR & \mbox{definite}\\ \hline
\end{array}
\]
 Putting all the above facts together, we get the classification in the first line of Table \ref{F:affineD}.
 \end{example}

 \begin{example} For $\alpha \in \bbR \setminus \{ 0, -1 \}$, $u = \alpha \ln(e^{2x}+1) + \ln(y)$ and $u = \alpha \ln(e^{2x}+1) + \ln(e^{2y}+1)$ in the 2nd and 3rd lines of Table \ref{F:affineD} are both affinely {\em inhomogeneous}, and the corresponding tubular CR structures are definite iff $\alpha < 0$ in the former case, while $\alpha > 0$ in the latter case.  These have respective tubular ILC structures:
 \[
 \begin{cases}
 w_{11} = -\frac{w_1(w_1-2\alpha)}{2\alpha},\\
 w_{12} = 0,\\
 w_{22} = -\frac{(w_2)^2}{2}
 \end{cases}, \qquad
 \begin{cases}
 w_{11} = -\frac{w_1(w_1-2\alpha)}{2\alpha},\\
 w_{12} = 0,\\
 w_{22} = -\frac{w_2(w_2-2)}{2}
 \end{cases}.
 \]
 The transformations $(\tilde{z}_1, \tilde{z}_2, \tilde{w}) = (e^{z_1}, z_2, -\frac{w}{2\alpha})$ and $(\tilde{z}_1, \tilde{z}_2, \tilde{w}) = (e^{z_1}, e^{z_2}, -\frac{w}{2\alpha})$ respectively map the above systems to (after dropping tildes) $w_{11} = (w_1)^2$, $w_{12} = 0$, $w_{22} = \lambda (w_2)^2$, where $\lambda = \alpha$.  As in Example \ref{Ex:D7}, this leads to $a = \frac{3}{4} (\frac{\alpha-1}{\alpha+1})$.  The explicit CR symmetry algebras (see Table \ref{F:affineD}) are isomorphic to $\fsl(2,\bbR) \times \fsu(2) \times \bbR$ and $\fsu(2) \times \fsu(2) \times \bbR$ respectively. In the latter case, the data obtained so far  (together with the table at the end of Example \ref{Ex:D7}) is sufficient to obtain the corresponding anti-involution classification in Table \ref{F:affineD}.  However, in the $\fsl(2,\bbR) \times \fsu(2) \times \bbR$ cases it is insufficient.  To do this, we need to find a Cartan basis (\S \ref{S:RealForms}) $\{ e_1, ..., e_7 \}$ with the CR symmetry algebra arising as the fixed point set of one of the anti-involutions in Table \ref{F:AI-reps}.
 
 First, we should work at a nice basepoint.  Apply a real affine transformation $(x,y,u) \mapsto (x, y - 1, u-\alpha x-y-\alpha \ln(2)+1)$, so $u = \alpha\ln(e^{2x}+1) + \ln( y )$ becomes
 \begin{align} \label{E:D7-case2}
 u = \alpha\ln(e^{2x}+1) + \ln(y+1) - \alpha x -  \alpha\ln(2) - y.
 \end{align}
 The corresponding tubular hypersurface has CR symmetry algebra
 \begin{align*}
 f_1 &= i \partial_{z_1}, \quad f_2 = i\partial_{z_2}, \quad f_3 = i\partial_w, \quad
 f_4 = (z_2+1)\partial_{z_2} - z_2 \partial_w, \quad f_5 = i\frac{z_2(z_2 + 2)}{2} \partial_{z_2} - i \frac{(z_2)^2}{2} \partial_w, \\
 f_6 &= \cosh(z_1) \partial_{z_1} + \alpha \sinh(z_1) \partial_w, \quad
 f_7 = i\sinh(z_1) \partial_{z_1} + i\alpha (\cosh(z_1) - 1) \partial_w.
 \end{align*}
These are also point symmetries for the corresponding ILC structure:
 \begin{align} \label{E:D7-case2-ILC}
 w_{11} = \frac{\alpha^2 - (w_1)^2}{2\alpha}, \quad w_{12} = 0, \quad w_{22} = -\frac{(w_2+1)^2}{2}.
 \end{align}
 The complexification of \eqref{E:D7-case2} has $w_1 = \frac{\partial w}{\partial z_1}$ and $w_2 = \frac{\partial w}{\partial z_2}$ vanishing at $(z_1,z_2,w) = (0,0,0)$ (and $(a_1,a_2,b) = (0,0,0)$).  At the basepoint $(z_1,z_2,w,w_1,w_2) = (0,0,0,0,0)$, we have $\fk = \langle f_5,\, f_7 \rangle$ and
 \[
  \fe = \langle E_1:= f_6 - i f_1,\, E_2:= f_4 - i f_2 \rangle + \fk, \quad \fv = \langle V_1:= f_6 + i f_1,\, V_2:= f_4 + i f_2 \rangle + \fk.
 \]
 Recalling that $a = \frac{3}{4} (\frac{\alpha-1}{\alpha+1} )$, we find that a Cartan basis is given by:
 \begin{align*}
 e_1 &= E_1, \quad
 e_2 = E_2 - i f_5, \quad
 e_3 = -\frac{1}{2} (4a-3) V_1, \quad
 e_4 = -\frac{1}{2} (4a+3) (V_2 + i f_5),\\
 e_5 &= i (4 a + 3) f_3+\frac{i}{2} (2a + 1) f_5 - \frac{i}{2}(2a-1) f_7, \quad
 e_6 = -i f_5 - i f_7, \quad
 e_7 = -i f_5 + i f_7.
 \end{align*}
 Although $i e_5, i e_6, i e_7$ are real, i.e.\ lie in the CR symmetry algebra, this Cartan basis is not aligned to the representative anti-involutions in Table \ref{F:AI-reps}.  We still need to use the residual basis change freedom $\diag(c_1, c_2, \frac{1}{c_1}, \frac{1}{c_2}, 1,1,1)$ preserving the ILC D.7 structure equations \cite[Table 4.3]{DMT2014}:
 \[
 \begin{array}{|c|c|c|c|c|c|} \hline
 a & \alpha & c_1 & c_2 & \mbox{Anti-involution $\varphi_1^{(\epsilon_1,\epsilon_2)}$} \\ \hline\hline
 a > \frac{3}{4} & \alpha < -1 & \sqrt{\frac{4a-3}{2}} & \sqrt{\frac{(-4a+3)}{2}\alpha} & \varphi_1^{(-1,-1)}\\
 a < -\frac{3}{4} & -1 < \alpha < 0 & \sqrt{\frac{3-4a}{2}} & \sqrt{\frac{4a-3}{2}\alpha} & \varphi_1^{(+1,+1)}
 \\
 |a| < \frac{3}{4} & 0 < \alpha < \infty & 
 \sqrt{\frac{3-4a}{2}} & \sqrt{\frac{(-4a+3)}{2}\alpha} & \varphi_1^{(+1,-1)} \\ \hline
 \end{array}
 \]
 (Recall that the parameter redundancy $a \mapsto -a$ induces the flip $\varphi_1^{(\epsilon_1,\epsilon_2)} \mapsto \varphi_1^{(\epsilon_2, \epsilon_1)}$, which explains why it is not necessary to list $\varphi_1^{(-1,+1)}$ above.)
 
 \end{example}

 \begin{example} For $\alpha\in \bbR$, $u = \alpha \arg(ix+y) + \ln(x^2+y^2)$ has corresponding tubular ILC structure:
 \[
 w_{11} = -w_{22} = \frac{w_1 w_2 \alpha - (w_1)^2 + (w_2)^2}{\alpha^2+4}, \quad
 w_{12} = \frac{((w_2)^2-(w_1)^2) \alpha - 4 w_1 w_2}{2(\alpha^2+4)}.
 \]
 The {\em complex} affine transformation $(\tilde{z}_1, \tilde{z}_2, \tilde{w}) = (z_2 + i z_1, z_1+i z_2, c w)$, where $c = -\frac{i\alpha+2}{\alpha^2+4}$, transforms this system (after dropping tildes) to the D.7 model $w_{11} = (w_1)^2, \, w_{12} = 0, \, w_{22} = \lambda (w_2)^2$ with $\lambda = \frac{2 + i\alpha}{2 - i\alpha}$.  As mentioned in Example \ref{Ex:D7}, $\lambda = \frac{3+4a}{3-4a}$, so we obtain $a = \frac{3}{8} i \alpha$.  From Table \ref{F:AI-reps}, we see that $\varphi_2$ gives rise to these CR structures (that have $\fsl(2,\bbC)_\bbR \times \bbR$ symmetry).  It is clear that $a \mapsto -a$ is a parameter redundancy here since the reflection $x\mapsto -x$ induces $\alpha \mapsto -\alpha$.
 \end{example}

Existence of a tubular representation can characterized at a Lie-algebraic level.

\begin{defn} A {\em tubular realization} for the complex homogeneous ILC structure $(\fs,\fk;\fe,\fv)$ in dimension $\dim(\fs/\fk) = 2n+1$ is a pair $(\fa,\varphi)$, where 
 \begin{enumerate}
 \item[(T.1)] $\fa \subset \fs$ is an $n$-dimensional abelian subalgebra;
 \item[(T.2)] the centralizer $\fc(\fa) = \{ X \in \fs : [X,Y] = 0,\, \forall Y \in \fa \}$ coincides with $\fa$ itself;
 \item[(T.3)] $\fa$ is transverse to both $\fe$ and $\fv$, i.e. $\fa \cap \fe = 0 = \fa \cap \fv$, so $\fa$ complements both $\fe$ and $\fv$ in $\fs$;
 \item[(T.4)] $\varphi$ is an admissible anti-involution of $(\fs,\fk;\fe,\fv)$ (see \S \ref{S:RealForms}) that preserves $\fa$.
 \end{enumerate}
\end{defn}

 For a tubular ILC structure (arising from \eqref{s4:surf_f}), the ILC symmetry algebra $\fs$ contains $\fa=\tspan_\bbC\{ \partial_{z_1}-\partial_{a_1}, ..., \partial_{z_n}-\partial_{a_n}\}$, which is $n$-dimensional abelian (T.1).  By Lemma \ref{L:proj}, we can identify $\fa$ with its projection $\langle \partial_{z_1}, ..., \partial_{z_n}\rangle$ to $M/V=\bbC^n$, and similarly its projection $\langle \partial_{a_1}, ..., \partial_{a_n}\rangle$ to $M/E=\bar\bbC^n$.  Every vector field on $\bbC^n$ commuting with all $\partial_{z_j}$ is itself an infinitesimal translation, so (T.2) holds.  The Lie group of translations acts transitively on $M/V=\bbC^n$ and $M/E=\bar\bbC^n$, which by dimension reasons forces $\fa  \cap \fv=0=\fa\cap \fe$, so (T.3) holds.  The anti-involution $\varphi$ of $\fs$ induced by $\tau$  \eqref{s4:anti-inv} preserves $\fa$, i.e. (T.4) holds, and $\fa^\varphi = \tspan_\bbR\{ i\partial_{z_j} \}$ (or more precisely, $i(\partial_{z_j} - \partial_{\bar{z}_j})$ here).
 
 The elements of $N(\fa)/\fa$ are in 1-1 correspondence with complex affine symmetries of~\eqref{s4:tube}. Indeed, let $X\in \fX(M/V)$ lie in $N(\fa)$, so $[\partial_{z_j},X]=A_j^k \partial_{z_k}$, where $A_j^k \in \bbC$.  Then $X= A^k_j z_k \partial_{z_j} \mod{\spn{\partial_{z_k}}}$, i.e. $X$ is complex-affine modulo $\fa$. Obviously any complex affine symmetry belongs to $N(\fa)$. Moreover since the second fundamental form of \eqref{s4:aff_eq} is non-degenerate, there are no symmetries of~\eqref{s4:tube} that are translations. This makes the correspondence with $N(\fa)/\fa$ one to one.
 
 Most of the CR structures in this article are tubular.  (Exceptions are discussed in \S \ref{S:Cartan} and \S \ref{S:Winkelmann}.)
 
\begin{theorem}\label{s4:t1}
Every multiply-transitive Levi non-degenerate hypersurface in $\bbC^3$ admits a tubular realisation unless it is a real form of  D.6-3 $(a^2 \in \bbR \backslash \{ 0, 9 \})$ or N-6.2 $(b^2 = \bar{a}^2 \in \bbC \backslash \bbR)$.
\end{theorem}

\begin{proof}
First, we show that the exceptional models do not have any tubular realisations. We can see this immediately for D.6-3 $(a^2\neq 9)$ since the 6-dimensional symmetry algebra in this case is semisimple and cannot have a 3-dimensional abelian subalgebra.

As discussed in \S \ref{S:RealForms}, the ILC N.6-2 models have essential parameters $(a^2,b^2) \in \bbC^2$.  These admit underlying CR structures when $b^2 = \bar{a}^2 \in \bbC$.  By Remark \ref{R:N62}, such structures are self-dual when $b^2 = a^2 \in \bbR$, which is a necessary condition for tubular realisability according to Proposition~\ref{s4:prop_SD}.

The existence of tubular realisations for all other cases is examined on the Lie algebra level. In Table \ref{s4:tbl_a} we list abelian subalgebras $\fa$ defining tubular realizations. Corresponding affine surfaces and symmetries of induced tubular CR hypersurfaces are listed in Tables \ref{F:affineN} and \ref{F:affineD}.
\end{proof}

 \begin{table}[h]
 {\tiny \[
 \begin{array}{|@{}c@{}|c|c|c|} \hline
 \mbox{Model} & \mbox{Anti-involution $\varphi$} & \mbox{$\varphi$-stable $\fa \subset \fs$} & \dim(N(\fa) / \fa)\\ \hline\hline
 \mbox{N.8} & \varphi & e_1-e_4,\,\, e_2 - e_3,\,\, e_5 & 2  \\ \hline
 \mbox{N.7-2} & \varphi^{(+1)} & e_1-e_4+2 e_6+\frac{1}{2} e_7,\,\, e_2 - e_3 - 2 e_7,\,\, e_5 + e_7 & 2 \\
 & \varphi^{(-1)} & e_1 + e_4, \,\, e_2 + e_3, \,\, e_5 + e_7 & 0\\ \hline
  \mbox{N.6-1}
  & \varphi & e_1-\frac1\tau e_4,\,\, e_2-\frac1\tau e_3 - \frac2a e_6,\,\, e_5 + \frac1{a^2+1} e_6 & \begin{array}{l} 1,\, \mbox{ for } a^2=2;\\
 2,\, \mbox{ for } a^2 \neq 2 \end{array}\\ \hline
 \begin{array}{c} \mbox{N.6-2}\\ b = a, \\ a^2 \in \bbR \backslash \{ -4 \} \end{array} & \varphi & e_1- e_4,\,\, e_2-e_3 - a e_6,\,\, e_5 - e_6 & 1\\ \hline
 \begin{array}{c} \mbox{N.6-2} \\ b = a = 2i \end{array} & \varphi & e_1+ e_4,\,\, e_2+e_3 - 2i e_6,\,\, e_5 + e_6  & 2\\ \hline
 \mbox{D.7} & \varphi_1^{(\epsilon_1,\epsilon_2)} &
\begin{array}{c} e_1 - \epsilon_1 e_3 -\sqrt{\frac{4 a-3}2}e_6 + \sqrt{\frac{4 a-3}2}e_7,\\
 e_2 - \epsilon_2 e_4  -\sqrt{\frac{4 a+3}2}e_6 - \sqrt{\frac{4 a+3}2}e_7,
 \\ e_5+\frac12 e_6 + a e_7 ;
\end{array} & 
  \begin{array}{l@{\,}l}
  1-\frac{\epsilon_1+\epsilon_2}2, & \mbox{ for } |a|<\frac34;
  \\[0.02in]
   1+\frac{\epsilon_1-\epsilon_2}2, &\mbox{ for }  a>\frac34;\\[0.02in]
   \frac{3-\epsilon_2}{2}, &\mbox{ for }  a=\frac34
  \end{array}
\\  \hline \mbox{D.7} & \varphi_2 & 
\begin{array}{c}
\frac2{3-4 a}e_1 + e_3 - e_6 + e_7,\,\,
e_2 + \frac2{3+4 a}e_4  + e_6 + e_7,
\\ e_5+\frac12 e_6 + a e_7 
\end{array}
& 2\\ \hline
 \mbox{D.6-1} &  \varphi^{(\epsilon)} & e_1 - e_3 + \frac{1}{\sqrt{2}} e_6, \,\, e_2 + e_4,\,\, e_5 + \frac{3}{4} e_6 & 2\\ \hline
 \mbox{D.6-2} &  \varphi^{(\epsilon)} & e_1 - \epsilon e_3,\,\, e_2-\tau e_4 +(a-\frac13),\,\, e_5+\frac{3a+5}{4(a-1)}e_6 & 2 \\ \hline
 \begin{array}{c} \mbox{D.6-3}\\ a^2=9 \end{array} & \varphi_1,\, \varphi_2^{(\epsilon)} & e_1 + \frac{3}{a} e_4,\,\, e_2 + \frac{3}{a} e_3,\,\, e_5 & 3\\ \hline
 \end{array}
 \]
 \caption{Tubular realizations $(\fa,\varphi)$}
 \label{s4:tbl_a}
 }
 \end{table}

\begin{example}[N.6-2 tubular cases] \label{EX:N62-tubular} In $\bbR^3$, consider the affine surface $u = y\exp(x) + \exp(\alpha x)$, for $\alpha \in \bbR \backslash \{ -1,0,1,2 \}$.  Proceeding similarly as in Example \ref{Ex:D7} yields the complex ILC structure
\[
 w_{11} = \frac{\alpha(\alpha-1) (w_2)^\alpha + w_1}{2}, \quad w_{12} = \frac{w_2}{2}, \quad w_{22} = 0.
\]
 Then $(\tilde{z}_1,\tilde{z}_2,\tilde{w}) = (\exp(\frac{z_1}{2}),z_2\exp(\frac{z_1}{2}),cw)$, where $c = 2^{\frac{1}{\alpha-1}}$, transforms the above system to
 \[
 \tilde{w}_{11} = \alpha(\alpha-1) (\tilde{z}_1)^{\alpha-2} (\tilde{w}_2)^\alpha, \quad \tilde{w}_{12} = \tilde{w}_{22} = 0.
 \]
 From \cite[Table 1.1]{DMT2014}, this is equivalent (for $\alpha \neq 0,1$) to an N.6-2 model with $\mu = \alpha$ and $\kappa = \alpha-2$. The parameters $(\mu,\kappa)$ are related to Cartan basis parameters $(a,b)$ by \cite[Table A.4]{DMT2014}:
 \begin{align} \label{E:N62-params}
 \mu = \frac{1}{2} + \frac{3b}{2\sqrt{b^2+4}}, \quad \kappa = -\frac{3}{2} + \frac{3a}{2\sqrt{a^2+4}}.
 \end{align}
 Hence, $\mu = \alpha$ and $\kappa = \alpha-2$ forces for $\alpha \in \bbR \backslash \{ -1,0,1,2 \}$ the relation displayed in Table \ref{F:affineN}:
 \begin{align} \label{E:N62-tube}
 b^2 = a^2 = \frac{-(2\alpha-1)^2}{(\alpha+1)(\alpha-2)}\in \bbR \backslash \left( [-4,0) \cup \left\{ \frac{1}{2} \right\}\right).
 \end{align}
 
For $u = xy + \exp(x)$, we get ILC structure $w_{11} = \frac{1}{2} e^{w_2}, \, w_{12} = \frac{1}{2}, \, w_{22} = 0$.  Then $(\tilde{z}_1,\tilde{z}_2,\tilde{w}) = (\frac{z_1}{2},\frac{z_2}{2},\frac{w}{2}- \frac{z_1z_2}{4})$ transforms it to the N.6-2 model in \cite[Table 1.1]{DMT2014} with $\mu = \kappa = \infty$, which is the same as $b^2 = a^2 = -4$ \cite[Table A.4]{DMT2014}.  Similarly, $u = y\exp(x) - \frac{x^2}{2}$ yields the system 
 $w_{11} = \frac{1}{2} (w_1 + \ln(w_2) - 1), \, w_{12} = \frac{1}{2} w_2,\, w_{22} = 0$.  Then $(\tilde{z}_1,\tilde{z}_2,\tilde{w}) = (e^{\frac{z_1}{2}}, \frac{z_2}{2}\exp(\frac{z_1}{2}), \frac{w}{2}+\frac{(z_1)^2}{8})$ transforms this to the N.6-2 model with $\mu=0, \kappa = -2$, i.e.\ $b^2 = a^2 = \frac{1}{2}$.
 
 It remains to describe the tubular cases corresponding to $b^2 = a^2 \in (-4,0]$.  For $\beta \in \bbR$, consider $u\cos(x) + y\sin(x) = \exp(\beta x)$.  Replacing $(x,y,u)$ by $(\frac{ix}{2},i(y+u),-y+u)$ maps this to $u e^{-x/2} - y e^{x/2} = e^{\beta i x/2}$, or $u = y e^x + e^{\alpha x}$ with $\alpha = \frac{\beta i + 1}{2}$.  We get to \eqref{E:N62-tube} as above, which yields $b^2 = a^2 = -\frac{4\beta^2}{\beta^2+9}$.  This yields the classification on the third line from the end of Table \ref{F:affineN}.
\end{example}
 
 \begin{example}
 In the D.6-2 case, consider the affine surface $u = y^2 + \epsilon x^\alpha$ for $x > 0$ and $\alpha \in \bbR \backslash \{ 0, 1, 2 \}$.  Following the procedure above, we get $\mu = \frac{\alpha-2}{\alpha-1}$, where $\mu$ is the parameter appearing in \cite[Table 1.1, D.6-2]{DMT2014}.  Since $\mu = \frac{6(a-1)}{3a-4}$ from  \cite[Table A.4]{DMT2014}, we get $a = \frac{2}{3}(\frac{\alpha+1}{\alpha}) \in \bbR \backslash \{ \frac{2}{3}, \frac{4}{3}, 1 \}$.  From Table \ref{F:AI-reps}, we have $\tau = -\frac{(3a-2)(a-1)}{9} = \frac{2(\alpha-2)}{27\alpha^2}$, and for $\rho = \pm 1$ the anti-involution $\varphi^{(\rho)}$ for D.6-2 in Table \ref{F:AI-reps} yields a definite structure if and only if $\rho \tau > 0$, i.e. $\rho(\alpha-2) > 0$. On the other hand, we find that the 2nd fundamental form of $u = y^2 + \epsilon x^\alpha$ is definite if and only if $\epsilon \alpha (\alpha-1) > 0$.  This forces $\rho = \epsilon$ for $\alpha \in (0,1) \cup (2,\infty)$ and $\rho = -\epsilon$ for $\alpha \in (-\infty,0) \cup (1,2)$, i.e.\ $\rho = \epsilon\,\sgn(\alpha(\alpha-1)(\alpha-2))$.  In terms of $a$, this is the same (after simplification) as $ \rho = \epsilon\,\sgn( (3a-2)(3a-4)(a-1))$.
 \end{example}

As in the above examples, we can consider other affine-homogeneous surfaces in $\bbR^3$ listed in \cite{DKR1995}.  According to \cite{DKR1995}, all of those with at least 3-dimensional affine symmetry algebra are given by cylinders with homogeneous base, quadrics, or the Cayley surface $u=xy-\frac{x^3}3$.  In almost all cases, these give surfaces with degenerate 2nd fundamental form or lead to flat CR structures.  The only exceptions are the (pseudo-)spheres $u^2 + \epsilon_1 x^2 + \epsilon_2 y^2 = 1$, where we can take $(\epsilon_1,\epsilon_2) \in \{ \pm (1,1),(1,-1) \}$.  These lead to complex ILC structures of type D.6-3 (with $a^2 = 9$ and symmetry algebra $\fso(3,\bbC) \ltimes \bbC^3$) and their corresponding CR real forms.  In all other cases in \cite{DKR1995}, the affine symmetry dimension is precisely 2.

\begin{remark} Among this remaining list in \cite{DKR1995}, some do not give multiply-transitive tubular hypersurfaces.  Namely, referring to the enumeration in \cite{DKR1995}:
\begin{itemize}
\item (5) with $\alpha\neq -1,0,8$, (16) with $\alpha\neq 4$ and (18) are of type I;
\item (1) and (2) are of type II; 
\item (5) with $\alpha=8$, (6) and (16) with $\alpha= 4$ are of type D with 5-dimensional symmetry algebra.
\end{itemize}
 Recall from Table \ref{F:AI-reps} that all the CR structures in our classification are of type N or D.
\end{remark}

 We conclude this section by showing that the dimensions of the affine symmetry subalgebras, i.e.\ $\dim(N(\fa)/\fa)$, listed in Table \ref{s4:tbl_a} are maximal among all possible tubular realizations $(\fa,\varphi)$.  It remains to prove this for those cases in Table \ref{s4:tbl_a} with $\dim(N(\fa)/\fa) \leq 1$.  See Tables \ref{F:affineN} \& \ref{F:affineD} for models.\\

\textbf{N.7-2.} Here, $\fs = \fsl(2,\bbC) \ltimes (V_2 \op V_0)$, where $V_k$ denotes the standard $(k+1)$-dimensional irreducible $\fsl(2,\bbC)$-module. (Here, $V_2 \op V_0$ is an abelian ideal in the symmetry algebra.)  Then $\varphi^{(+1)}$ and $\varphi^{(-1)}$ (see Table \ref{F:AI-reps}) lead to CR structures with $\fsl(2,\bbR) \ltimes (V_2 \op V_0)$ and $\fsu(2) \ltimes (V_2 \op V_0)$ symmetry respectively.\footnote{We have abused notation here: $V_2$ refers to the adjoint representation of $\fsl(2,\bbR)$ or $\fsu(2)$ respectively, and $V_0$ is the trivial 1-dimensional representation.}  It suffices to consider the latter case.  We will show $N(\fa)=\fa$ always.

 Since $\fa \subset \fs$ is self-centralizing, then $\fa \not\subset V_2\op V_0$ or else $\fc(\fa) \supset V_2\op V_0$.  The projection of $\fa$ on $\fsu(2)$ must be 1-dimensional. Consider $x+v\in \fa$, where $0 \neq x\in \fsu(2)$ and $v\in V_2\op V_0$. From the centralizer of $x+v$, we get $\fa=\langle 
x+v,v_0,v_1\rangle$, where $v_0$ spans $V_0$ and $0 \neq v_1\in V_2$ is in the kernel of $x$.  Since $x+v$ is semisimple and has 1-dimensional kernels on $\fsu(2)$ and $V_2$, then $\dim(N(\fa)) = 3$.\\

\textbf{D.7.} When $a\neq\pm\frac{3}{4}$, the symmetry algebra is  $\fs=\fsl(2,\bbC)\times \fsl(2,\bbC)\times \bbC$.  All real forms of $\fs$ are determined by real forms of the semisimple part, namely $\fsl(2,\bbR)\times \fsl(2,\bbR)$, $\fsl(2,\bbR)\times \fsu(2)$, $\fsu(2)\times \fsu(2)$, and $\fsl(2,\bbC)_\bbR$. Any 3-dimensional abelian subalgebra $\fa \subset \fs$ is generated by the center and one element $T_j$ from each copy of $\fsl(2,\bbC)$, and $N(\fa)$ is the  intersection of the normalizers of $T_j$.  The element $T_j$ has a 2-dimensional normalizer in the corresponding copy of $\fsl(2,\bbC)$ if it is nilpotent and 1-dimensional otherwise.  Since $\fsu(2)$ consists of semisimple elements, we immediately see that $\dim( N(\fa)) \leq 5$ for the $\fsl(2,\bbR)\times \fsl(2,\bbR)$ and $\fsl(2,\bbC)_\bbR$ cases,  $\dim( N(\fa)) \leq 4$ for the $\fsl(2,\bbR)\times \fsu(2)$ case, and finally $\dim( N(\fa))=3$ for the $\fsu(2)\times \fsu(2)$ case.

 When $a=\pm\frac34$, $\fs = \fsl(2,\bbC) \times \fr$, where $\fr$ has relations $[S,X]=X,\,\, [S,Y]=-Y,\,\, [X,Y]=Z$.  Since $\fc(\fa)=\fa$, and $\fr$ contains no 3-dimensional abelian subalgebra, then $\fa$ must be spanned by the central element $Z$, a non-central element $R \in \fr$, and some $T \in\fsl(2,\bbC)$. The normalizer in $\fr$ of $R$ is at most 3-dimensional in $\fr$, and the normalizer in $\fsl(2,\bbC)$ of $T$ has dimension 2 if $T$ is nilpotent and 1 if $T$ is semisimple. Hence, if the real form of $\fs$ contains $\fso(3)$ (namely, for $\varphi^{\epsilon_1,1}$), then $\dim(N(\fa)) \leq 4$.\\

\textbf{N.6-2.}  Tubular CR structures arise when $b^2 = a^2 \in \bbR$.  Using the parameter redundancy (see \S\ref{S:RealForms}), we can always assume that $b = a$. In the generic case, $b^2 = a^2 \in \bbR \backslash \{ -4, \frac{1}{2} \}$, consider the basis of $\fs$ from \cite[Table A.4]{DMT2014}.  This satisfies $\kappa = \mu-2$ and the commutator relations in \cite[Table A.1]{DMT2014}:
 \[
     \begin{array}{c|cccccccc} 
     & S_1 & S_2 & N_1 & N_2 & N_3 & N_4\\ \hline
     S_1 & \cdot & \cdot & (\mu-1)N_1  & \mu N_2 &  \mu N_3 & (\mu-1)N_4 \\ 
     S_2 & & \cdot & (\mu-1)N_1 & (\mu-1)N_2 & \mu N_3 & \mu N_4 
     \end{array}
 \]
 with $\fn= \langle N_1,N_2,N_3,N_4\rangle$ an abelian ideal.  From \eqref{E:N62-params}, we have $\mu = \frac{1}{2} + \frac{3a}{2\sqrt{a^2+4}} \in \bbC \setminus \{0,1 \}$, and according to \cite[Table A.5]{DMT2014} models parametrised by $\mu$ and $1-\mu$ are equivalent.

To show $\dim(N(\fa)) \leq 4$ for any affine realization $\fa \subset \fs$, it suffices to show $\dim(N(\fa) \cap \fn) \leq 2$.  First note that $\fa \not\subset \fn$, since otherwise its centralizer would contain $\fn$ (4-dimensional). Therefore $\fa$ must contain an element $T = \alpha S_1 + \beta S_2 + v$, where $v \in \fn$ and $(\alpha,\beta) \neq (0,0)$.  Since $\fa$ is abelian, then $(\Ad_T|_{N(\fa)})^2 = 0$.  Since $\Ad_T|_\fn$ is diagonalizable, then $\ker(\Ad_T|_\fn)^2 = \ker(\Ad_T|_\fn)$.  Hence, $N(\fa) \cap \fn \subset \ker(\Ad_T|_\fn)$.  We want $\dim(\ker(\Ad_T|_\fn)) \leq 2$.  The eigenvalues of $\Ad_T|_\fn$ are
 \begin{equation} \label{E:eqn_mu}
 (\alpha+\beta)(\mu-1), \quad
 (\alpha+\beta)\mu -\beta,\quad
 (\alpha+\beta)\mu,\quad
 (\alpha+\beta)\mu - \alpha,
 \end{equation}
 and $\dim(\ker(\Ad_T|_\fn)) \ge 3$ would contradict $(\alpha,\beta) \neq (0,0)$.  Thus, $\dim(N(\fa) \cap \fn) \leq 2$ follows.
 
The two remaining cases are $a^2=b^2=-4$ and $a^2=b^2=\frac{1}{2}$. The former admits an affinely homogeneous tubular representation, while the latter has structure constants (see \cite[Table A.4]{DMT2014}):
\[  \begin{array}{c|cccccccc} 
     & S_1 & S_2 & N_1 & N_2 & N_3 & N_4\\ \hline
     S_1 & \cdot & N_3 & -N_1  & \cdot &  \cdot & -N_4 \\ 
     S_2 & & \cdot & -N_1 & -N_2 & \cdot & \cdot \end{array} \]
 (Again, $\fn= \langle N_1,N_2,N_3,N_4\rangle$ is an abelian ideal.)  The condition $\dim(\ker(\Ad_T|_\fn)) \leq 2$ easily follows.\\

\textbf{N.6-1 $(a^2=2)$.}
 From \cite[Table A.1]{DMT2014}, we have abelian $\fn= \langle N_2,N_3,N_4,N_5\rangle$ and commutators
\[  \begin{array}{c|cccccccc} 
     & S & N_1 & N_2 & N_3 & N_4 & N_5\\ \hline
     S & \cdot & N_1 - N_2 & N_2  & 2N_3 &  3 N_4 & 2 N_5 \\ 
     N_1 & & \cdot & N_3 & N_4 & \cdot & \cdot \end{array} \]
 As above, $\fa \not\subset \fn$, and $\fa$ contains $T=\alpha S+\beta N_1 + v$, where $v\in \fn$ and $(\alpha,\beta) \neq (0,0)$.  Since $\dim(\fa) = 3$, then $\dim(\fa \cap \fn) \geq 1$, and $\Ad_T|_{\fa \cap \fn} = 0$. This forces $\alpha=0$.  Hence, $\fa$ must be spanned by $N_1 + t_2 N_2 + t_3 N_3, N_4,N_5$.  Note $N_3 \in N(\fa)$, so $\dim(N(\fa)) \geq 4$.  But $\dim(N(\fa)) \leq 4$, since
 \begin{align*}
 [\gamma S+\delta N_2, N_1+t_2 N_2 + t_3 N_3] 
 &= \gamma (N_1+(t_2-1) N_2 + 2 t_3 N_3)  - \delta N_3\\
 &\equiv \gamma (-N_2 + t_3 N_3)  - \delta N_3 \quad\mod \fa,
 \end{align*}
 indicates that $\gamma S + \delta N_2 \in N(\fa)$ only when $\gamma = \delta = 0$.

\section{Real forms of ILC D.6-3 models}
\label{S:Cartan}
 
 \subsection{Real form symmetry algebras}
 
 The ILC D.6-3 models \cite[Table 4.3]{DMT2014} admit a single essential parameter $a^2 \in \bbC \backslash \{ 0 \}$.  In the Cartan basis, the D.6-3 structure equations are:
 \begin{align} \label{E:D63-streq}
 \begin{array}{c|cccccccc}
 & e_1 & e_2 & e_3 & e_4 & e_5 & e_6\\ \hline
 e_1 & \cdot & \frac{a}{2} e_6 & -e_5 - \frac{3}{2} e_6 & \cdot & -\frac{3}{2} e_1 - \frac{a}{2} e_4 & -e_1\\
 e_2 & & \cdot & \cdot & -e_5 + \frac{3}{2} e_6 & -\frac{3}{2} e_2 - \frac{a}{2} e_3 & e_2\\
 e_3 & & & \cdot & -\frac{a}{2} e_6 & +\frac{3}{2} e_3 + \frac{a}{2} e_2 & e_3\\
 e_4 & & & & \cdot & +\frac{3}{2} e_4 + \frac{a}{2} e_1  & -e_4\\
 e_5 & & & & & \cdot & \cdot\\
 e_6 & & & & & & \cdot
 \end{array}
 \end{align}
 When $a^2=9$ (labelled D.6-3$_\infty$ in \cite{DMT2014}), the symmetry algebra is $\fs \cong \fso(3,\bbC) \ltimes \bbC^3$, the model admits a tubular representation, and all CR real forms are given in Table \ref{F:affineD}.  (The 2nd fundamental form for $u^2 + \epsilon_1 x^2 + \epsilon_2 y^2 = 1$ has definite signature if and only if $\epsilon_1 \epsilon_2 > 0$, so $(\epsilon_1,\epsilon_2) = (1,-1)$ corresponds to $\varphi_1$ since this yields an indefinite structure (Table \ref{F:AI-reps}).  The $\varphi_2^{(\pm 1)}$ cases are then identified from \eqref{E:D63-streq} and the semisimple part of the symmetry algebra.)
 
All models with $a^2 \in \bbC \backslash \{ 0, 9 \}$ are all non-tubular with symmetry algebra $\fs \cong \fsl(2,\bbC) \times \fsl(2,\bbC) \cong \fso(4,\bbC)$.  Recall that the real forms of $\fso(4,\bbC)$ and their Killing form signatures are
 \begin{align*}
 \fso(4) \cong \fsu(2) \times \fsu(2): \quad (0,6); \\
 \fso(3,1) \cong \fsl(2,\bbC)_\bbR: \quad (3,3);\\
 \fso(2,2) \cong \fsl(2,\bbR) \times \fsl(2,\bbR): \quad (4,2);\\
 \fso^*(4) \cong \fsl(2,\bbR) \times \fsu(2): \quad (2,4).
 \end{align*}
 Using our list of the D.6-3 admissible anti-involutions $\varphi$ (Table \ref{F:AI-reps}), we obtain a basis of the real form $\fs^\varphi$ of $\fso(4,\bbC)$,  and classify $\fs^\varphi$ from the signature of its Killing form (Table \ref{F:D63-CR-sym}).  Note that these CR structures only arise when $a^2 \in \bbR \backslash \{ 0, 9 \}$ (Theorem \ref{s4:t1}).
  
 \begin{table}[h]
 \[
 \begin{array}{|cccc|} \hline
 \mbox{Anti-involution} & \mbox{$a^2$ range} & 
 \mbox{Real form of $\fso(4,\bbC)$} & \mbox{Levi-form type}\\ \hline\hline
 \varphi_1 & \begin{array}{cl}
 0 < a^2 < 9 \\
 a^2 > 9
 \end{array} &
 \begin{array}{cl}
 \fso(3,1)  \\
 \fso(2,2)
 \end{array}
 & \mbox{indefinite}\\ \hline
 \varphi_2^{(+1)} &
 \begin{array}{cl}
 0 < a^2 < 9 \\
 a^2 > 9
 \end{array} &
 \begin{array}{cl}
 \fso(4)  \\
 \fso(3,1)
 \end{array}
 & \mbox{definite}\\ \hline
 \varphi_2^{(-1)} &
 \begin{array}{cl}
 0 < a^2 < 9 \\
 a^2 > 9
 \end{array} &
 \begin{array}{cl}
 \fso(2,2)  \\
 \fso(3,1)
 \end{array}
 & \mbox{definite}\\ \hline
 \varphi_3 & a^2 < 0 & \fso^*(4) & \mbox{indefinite}\\ \hline
 \end{array}
 \]
 \caption{CR structures underlying ILC D.6-3 models with $\fso(4,\bbC)$-symmetry}
 \label{F:D63-CR-sym}
 \end{table}

 \subsection{Cartan hypersurfaces}
 \label{S:Cartan-hyp}
 
 Let $(\cdot,\cdot)$ denote a non-degenerate symmetric bilinear form on $\bbC^4$.  The Lie group $\tO(4,\bbC)$ preserves $\cQ = \{ [z] : (z,z) = 0 \} \subset \bbC\bbP^3$ and acts transitively on $\bbC\bbP^3 \backslash \cQ$.  Define $A = \begin{psmallmatrix} (z,z) & (z,\bar{z}) \\ (z,\bar{z}) & (\bar{z},\bar{z}) \end{psmallmatrix}$.  The (complex) scaling $z \mapsto \lambda z$ induces $A \mapsto LA\bar{L}$, where $L = \diag(\lambda,\bar\lambda)$.  On $\bbC\bbP^3 \backslash \cQ$, this scaling action has invariant $\alpha := \frac{(z,\bar{z})^2}{(z,z)(\bar{z},\bar{z})}$.  We will fix $(\cdot,\cdot)$ that is non-degenerate and $\bbR$-valued on $\bbR^4 \subset \bbC^4$, so that $(\bar{z},\bar{z}) = \overline{(z,z)}$.  In this case, $\beta := \frac{(z,\bar{z})}{|(z,z)|} \in \bbR$ is invariant under complex scalings, and $\alpha = \beta^2$.  We refer to the real hypersurfaces of $\bbC\bbP^3 \backslash \cQ$ uniformly described by 
 \begin{align} \label{E:Cartan-hyp}
 (z,\bar{z}) = \beta|(z,z)|
 \end{align}
 as {\em Cartan hypersurfaces}.  A precise list of inequivalent such structures is given in Table \ref{F:Cartan-hyp}.  Restricting $z = (z_1,...,z_4)^\top$ in \eqref{E:Cartan-hyp} to the standard affine coordinate chart $z_1 = 1$ on $\bbC\bbP^3$ recovers Loboda's models \cite[eqns (2.8)\&(2.9)]{Loboda2001a}, \cite[eqns (6)\&(7)]{Loboda2003}, which are generalizations of models found by Cartan \cite[eqn (10)]{Cartan1932b}.  In this subsection, we match these models with their corresponding ILC structures and identify the associated anti-involutions.
 
 Given $z = (z_1,...,z_4)^\top$, consider $(z,z) = \epsilon_1 (z_1)^2 + ... + \epsilon_4 (z_4)^2$, where $\epsilon_j = \pm 1$.  Then $\fso(4,\bbC)$ can be identified with the $\bbC$-span of $\sfZ_{jk} = \epsilon_j z_j \partial_{z_k} - \epsilon_k z_k \partial_{z_j}$, where $1 \leq j < k \leq 4$.  Their $\bbR$-span is identified with the real form of $\fso(4,\bbC)$ that preserves the restriction of $(\cdot,\cdot)$ to $\bbR^4 \subset \bbC^4$.  Let us classify real 5-dimensional orbits of each of $\tO(4),\tO(3,1),\tO(2,2)$ on $\bbC\bbP^3 \backslash \cQ$.  (Here, we take $\tO(3,1)$ corresponding to $\epsilon_1 = \epsilon_2 = \epsilon_3 = +1$ and $\epsilon_4 = -1$, etc.)  Given $[z]$ in such an orbit, $v = \Re(z)$ and $w = \Im(z)$ span a real 2-plane $\Pi$ in $\bbR^4$ (otherwise the orbit is only 3-dimensional). Since $(c_1 + c_2 i)(v + w i) = c_1 v - c_2 w + (c_2 v + c_1 w)i$, then the new real and imaginary parts satisfy
  \begin{align}
 (c_1 v - c_2 w, c_2 v + c_1 w) = c_1c_2 ((v,v) - (w,w)) + (c_1^2 - c_2^2) (v,w).
 \end{align}
 This quadratic in $c_1,c_2$ has discriminant $((v,v) - (w,w))^2 + 4(v,w)^2 \geq 0$.  Thus, using complex multiplication, we may assume that $(v,w) = 0$. Hence, $0 \neq (z,z) = (v+wi,v+wi) = (v,v) - (w,w)$ so that only real ($c_2 = 0$) or purely imaginary ($c_1 = 0$) rescalings preserve this orthogonality.

 Suppose that $(\cdot,\cdot)|_\Pi$ is non-degenerate.  Applying each real orthogonal group and the residual rescalings to the orthogonal pair $\{ v,w \}$, we obtain the following normal forms for $[z] \in \bbC\bbP^3 \backslash \cQ$:
 \begin{enumerate}
 \item $\Pi$ positive-definite: $z = (1,iy,0,0)$, where $0 < y < 1$.  Then $\beta = \frac{1+y^2}{1-y^2} > 1$.
 \item $\Pi$ indefinite: assuming signature $(+++\,-)$ or $(++-\,-)$, we have
 \begin{itemize}
 \item $z = (1,0,0,iy)$, where $0 < y < 1$ $\qRa \beta = \frac{1-y^2}{1+y^2}$ satisfies $0 < \beta < 1$;
 \item $z = (iy,0,0,1)$, where $0 < y < 1$ $\qRa \beta = -\frac{1-y^2}{1+y^2}$ satisfies $-1 < \beta < 0$.
 \end{itemize}
 \end{enumerate}
 The negative-definite case can be made positive-definite by absorbing a sign in \eqref{E:Cartan-hyp} into $\beta$.
 
 For a given orbit $M$ with basepoint a normal form $o = [z]$ above, the standard complex structure $J$ on $\bbC\bbP^3$ induces $H = TM \cap J(TM)$ and a CR structure.  For each such $M$, we determine a Cartan basis $e_1,..., e_6$ of $\fso(4,\bbC)$, i.e.\ so that the ILC D.6-3 structure equations are satisfied.  (Each basis element will be a $\bbC$-linear combination of $\sfZ_{ij}$.)  We begin by choosing a generator $e_6$ for the (complex) isotropy $\fk \subset \fso(4,\bbC)$ so that $\ad(e_6)$ has eigenvalues $(+1,-1,-1,+1)$ on $(\fe+\fv)/\fk$ (which corresponds to $H$).  We then choose a basis $e_1,...,e_4$ adapted to the $\pm i$-eigenspaces for $J$, then rescale the basis so that $[e_1,e_2] = -[e_3,e_4] \in \fk$, and finally choose $e_5$ to satisfy the remaining structure equations.  We summarize the results below.  In each case, with respect to the basis $e_1,...,e_4 \mod \fk$, we have $J = \diag\left(
 \begin{psmallmatrix} 0 & -y\\ \frac{1}{y} & 0 \end{psmallmatrix}, 
 \begin{psmallmatrix} 0 & -y\\ \frac{1}{y} & 0 \end{psmallmatrix}\right)$.  
 We also define $\sfv_1,\sfv_2$ so that 
 \begin{align} \label{E:EV-eigen}
 \begin{cases}
 +i \mbox{ eigenspace } \fe/\fk &= \tspan\{ \overline{\sfv_1}, \overline{\sfv_2} \} \mod \fk,\\
 -i \mbox{ eigenspace } \fv/\fk &= \tspan\{ \sfv_1, \sfv_2 \} \mod \fk.
 \end{cases}
 \end{align}
 \begin{enumerate}
 \item $\Pi$ positive-definite for $\tO(4),\tO(3,1),\tO(2,2)$: $o = [(1,iy,0,0)^\top]$, $0 < y < 1$.  Isotropy: $\sfZ_{34}$.
 \begin{align*}
 &\sfZ_{13}|_o = \partial_{z_3}, \quad \sfZ_{23}|_o = iy \partial_{z_3}, \quad
 \sfZ_{14}|_o = \partial_{z_4}, \quad \sfZ_{24}|_o = iy \partial_{z_4} \quad\mbox{(contact subspace)};\\
 & \begin{cases}
 \sfv_1 := y \sfZ_{13} + i \sfZ_{23}, \\
 \sfv_2 := y \sfZ_{14} + i \sfZ_{24}\\
 \end{cases} \quad a = \frac{3(1-y^2)}{1+y^2} \qRa \framebox{$0 < a^2 < 9$, \quad $\beta > 1$}.
 \end{align*}
 \begin{itemize}
 \item $\tO(4)$-case:\,\,\quad $\rho = \sqrt{\frac{3}{4(1+y^2)}}$, $\begin{cases}
 e_1 = \rho(\overline{\sfv_1} - i \overline{\sfv_2}), \,
 e_2 = \rho(\overline{\sfv_1} + i \overline{\sfv_2}), \\
 e_3 = \rho (\sfv_1 + i \sfv_2),\,
 e_4 = \rho (\sfv_1 - i \sfv_2), \\
 e_5 = \frac{3iy}{1+y^2} \sfZ_{12},\, e_6 = i\sfZ_{34}
 \end{cases}$
 \item $\tO(3,1)$-case: $\rho = \sqrt{\frac{3}{4(1+y^2)}}$, $\begin{cases}
 e_1 = \rho(\overline{\sfv_1} -\overline{\sfv_2}), \,
 e_2 = \rho(\overline{\sfv_1} + \overline{\sfv_2}), \\
 e_3 = \rho (\sfv_1 + \sfv_2),\,
 e_4 = \rho (\sfv_1 - \sfv_2), \\
 e_5 = \frac{3iy}{1+y^2} \sfZ_{12},\, e_6 = \sfZ_{34}
 \end{cases}$
 \item $\tO(2,2)$-case: $\rho = \sqrt{\frac{-3}{4(1+y^2)}}$, $\begin{cases}
 e_1 = \rho(\overline{\sfv_1} + i\overline{\sfv_2}), \,
 e_2 = \rho(\overline{\sfv_1} - i\overline{\sfv_2}), \\
 e_3 = \rho (\sfv_1 - i \sfv_2),\,
 e_4 = \rho (\sfv_1 + i \sfv_2), \\
 e_5 = \frac{3iy}{1+y^2} \sfZ_{12},\, e_6 = i\sfZ_{34}
 \end{cases}$
 \end{itemize}

 \item $\Pi$ indefinite for $\tO(3,1),\tO(2,2)$: $o = [(1,0,0,iy)^\top]$, $0 < y < 1$.  Isotropy: $\sfZ_{23}$.
 \begin{align*}
 &\sfZ_{12}|_o = \partial_{z_2}, \quad
 \sfZ_{24}|_o = iy \partial_{z_2}, \quad
 \sfZ_{13}|_o = \partial_{z_3}, \quad 
 \sfZ_{34}|_o = iy\partial_{z_3} \quad\mbox{(contact subspace)};\\
 & \begin{cases}
 \sfv_1 := y \sfZ_{12} + i \sfZ_{24},\\
 \sfv_2 := y \sfZ_{13} + i \sfZ_{34}\\
 \end{cases} \quad a = \frac{3(1+y^2)}{1-y^2} \qRa \framebox{$a^2 > 9$, \quad $0 < \beta < 1$}
 \end{align*}
 \begin{itemize}
 \item $\tO(3,1)$-case: $\rho = i\sqrt{\frac{3}{4(1-y^2)}}$, $\begin{cases}
 e_1 = \rho(\overline{\sfv_1} - i \overline{\sfv_2}),\,
 e_2 = \rho(\overline{\sfv_1} + i \overline{\sfv_2}),\\
 e_3 = \rho(\sfv_1 + i \sfv_2),\, 
 e_4 = \rho(\sfv_1 - i \sfv_2),\\
 e_5 = \frac{3iy}{1-y^2} \sfZ_{14},\, e_6 = i \sfZ_{23}
 \end{cases}$\\
Anti-involution: $\varphi_2^{(-1)}$, since $\fso(3,1) = \tspan_\bbR\{ i(e_1 + e_3), e_1 - e_3, i(e_2 + e_4), e_2 - e_4, i e_5, ie_6 \}$.
 \item $\tO(2,2)$-case: $\rho = \sqrt{\frac{3}{4(y^2-1)}}$, $\begin{cases}
 e_1 = \rho(\overline{\sfv_1} - \overline{\sfv_2}),\,
 e_2 = \rho(\overline{\sfv_1} + \overline{\sfv_2}),\\
 e_3 = \rho(\sfv_1 + \sfv_2),\, e_4 = \rho(\sfv_1 - \sfv_2),\\
 e_5 = \frac{3iy}{1-y^2} \sfZ_{14},\, e_6 = \sfZ_{23}
 \end{cases}$
 \end{itemize}
 \item $\Pi$ indefinite for $\tO(3,1)$: $o = [(iy,0,0,1)^\top]$, $0 < y < 1$.  Isotropy: $\sfZ_{23}$.
 \begin{align*}
 &\sfZ_{24}|_o = \partial_{z_2}, \quad
 \sfZ_{12}|_o = iy \partial_{z_2}, \quad
 \sfZ_{34}|_o = \partial_{z_3}, \quad 
 \sfZ_{13}|_o = iy\partial_{z_3} \quad\mbox{(contact subspace)};\\
 & \begin{cases}
 \sfv_1 := y \sfZ_{24} + i \sfZ_{12}, \\
 \sfv_2 := y \sfZ_{34} + i \sfZ_{13}\\
 \end{cases} \quad a = \frac{3(1+y^2)}{1-y^2} \qRa \framebox{$a^2 > 9$, \quad $-1 < \beta < 0$}
 \end{align*}
 \begin{itemize}
  \item $\tO(3,1)$-case: $\rho = \sqrt{\frac{3}{4(1-y^2)}}$, 
 $\begin{cases}
 e_1 = \rho(\overline{\sfv_1} - i \overline{\sfv_2}),\,
 e_2 = \rho(\overline{\sfv_1} + i \overline{\sfv_2}),\\
 e_3 = \rho(\sfv_1 + i \sfv_2),\,
 e_4 = \rho(\sfv_1 - i \sfv_2),\\
 e_5 = \frac{3iy}{1-y^2} \sfZ_{14},\, e_6 = i\sfZ_{23}
 \end{cases}$\\
 Anti-involution: $\varphi_2^{(+1)}$, since $\fso(3,1) = \tspan_\bbR\{ e_1 + e_3, i(e_1 - e_3), e_2 + e_4, i(e_2 - e_4), i e_5, ie_6 \}$.  By flipping the signature, we can write this as an $\tO(1,3)$ case with $0 < \beta < 1$.
 
 \item $\tO(2,2)$-case: Flipping the signature reduces this to the earlier $\tO(2,2)$ case.
 \end{itemize}
 \end{enumerate}
 
 A posteriori, we have \framebox{$a^2 = \frac{9}{\alpha} = \frac{9}{\beta^2}$}, so all of Table \ref{F:D63-CR-sym} is covered except for the models arising from $\varphi_3$.  The complete list of inequivalent Cartan hypersurfaces \eqref{E:Cartan-hyp} is given in Table \ref{F:Cartan-hyp}.  This classification is consistent with that of Loboda \cite[eqns (2.8)\&(2.9)]{Loboda2001a}, \cite[eqns (6)\&(7)]{Loboda2003}.
 
 \begin{table}[h]
 \[
 \begin{array}{|ccccc|} \hline
 \mbox{Real form} & \mbox{Signature of $(\cdot,\cdot)|_{\bbR^4}$} & \mbox{$\beta$-range} & \mbox{Levi-form type} & \mbox{Anti-involution}\\ \hline
 \tO(4) & ++++ & \beta > 1 & \mbox{definite} & \varphi_2^{(+1)}\\
 \tO(3,1) & +++- & 0 < \beta < 1 & \mbox{definite} & \varphi_2^{(-1)}\\
 && \beta > 1 & \mbox{indefinite} & \varphi_1\\
 \tO(1,3) & +--- & 0 < \beta < 1 & \mbox{definite} & \varphi_2^{(+1)}\\
 \tO(2,2) & ++-- & 0 < \beta < 1 & \mbox{indefinite} & \varphi_1\\
 && \beta > 1 & \mbox{definite} & \varphi_2^{(-1)}\\ \hline
 \end{array}
 \]
 \caption{Inequivalent Cartan hypersurfaces $\{ [z] : (z,\overline{z}) = \beta|(z,z)| \} \subset \bbC\bbP^3 \backslash \cQ$, where $(z,z) = \epsilon_1 z_1^2 + ... + \epsilon_4 z_4^2$ with $\epsilon_j = \pm 1$.  These complexify to the ILC D.6-3 model with $a^2 = \frac{9}{\beta^2}$.}
 \label{F:Cartan-hyp}
 \end{table}

 \subsection{Quaternionic models} 
 
 It remains to describe the CR structures associated with the anti-involution $\varphi_3$, which are all indefinite type and admit $\fso^*(4)$ symmetry.  To our knowledge, these models are new.  While $\fso^*(4)$ is customarily defined via a skew-Hermitian form $\eta$ on $\bbH^2$, where $\bbH$ is the quaternions, we will instead focus on the special isomorphism $\fso^*(4) \cong \fsl(2,\bbR) \times \fsu(2)$.  In doing so, we will work with a particularly simple representation of this Lie algebra and the choice of $\eta$ will arise naturally from this.  In contrast, if we were to fix a choice of $\eta$ first, the resulting realization of $\fso^*(4)$ could be quite complicated.
  
 Recall that $\bbH$ is the associative $\bbR$-algebra with $\bbR$-basis $1,\bi,\bj,\bk$, standard relations $\bi\bj = \bk,\, \bi^2 = \bj^2 = \bk^2 = -1$, conjugation satisfying $\overline{q_1 q_2} = \overline{q_2} \,\overline{q_1}$, and norm $|q|^2 = q\overline{q} = \overline{q} q$.  Let $\SU(2)$ denote the unit quaternions, which act on $\bbH$ on the left.  Let $\widehat{\SL}(2,\bbR)$ denote $2\times 2$ real matrices with determinant $\pm 1$, which act on $\bbR^2$.  The group $S = \widehat\SL(2,\bbR) \times \SU(2)$ acts on the external tensor product $\bbR^2 \otimes_\bbR \bbH$, and we identify this naturally with $\bbH^2$.  This identification is $S$-equivariant if for $q = \begin{psmallmatrix} q_1\\ q_2 \end{psmallmatrix} \in \bbH^2$, we declare that $A \in \widehat{\SL}(2,\bbR)$ and $q_0 \in \SU(2)$ each act by multiplication on the left, with the latter identified with $\diag(q_0,q_0)$.  (These actions commute.)  While $\bbH^2$ is naturally a {\em right} $\bbH$-vector space, it will be more important for us to consider it as a $\bbC$-vector space by restricting this right action to $\bbC := \{ 1s + \bi t : s,t \in \bbR\} \subset \bbH$.  (This in particular distinguishes the imaginary unit $\bi$.)  We will be interested in the 5-dimensional $S$-orbits in $\bbC\bbP^3 \cong \bbP_\bbC(\bbH^2)$.
  
 There exists an $S$-invariant skew-Hermitian form on $\bbH^2$ (unique up to a real scaling) given by $\eta(q,w) = \overline{q_1} w_2 - \overline{q_2} w_1$, and this is valued in $\Im(\bbH)$.  Let $\eta(q,q) = \bi b + \bj \mu$, for $b \in \bbR$ and $\mu \in \bbC$.  Given $\lambda \in \bbC$, $\eta(q\lambda,q\lambda) = \bar\lambda\eta(q,q)\lambda$, hence $(b,\mu) \mapsto (b|\lambda|^2, \mu\lambda^2)$.  
 Writing $q = \begin{psmallmatrix} z_1 + \bj z_2\\ z_3 + \bj z_4 \end{psmallmatrix}$ for $z_j \in \bbC$, 
 \begin{align}
 \eta(q,q) &= \overline{z_1} z_3 + \overline{z_2} z_4 + \bj(z_1 z_4 - z_2 z_3).
 \end{align}
 Note that $\cQ^\#= \{ [q] : \Re(\bi\eta(q,q)) = 0 \}$ is the flat (indefinite) CR structure $\Im(\overline{z_1} z_3 + \overline{z_2} z_4) = 0$, so we exclude $b=0$.  We will also exclude the $\mu = 0$ case (see below), since this yields only 4-dimensional $S$-orbits.  When $b\mu \neq 0$, we have the complex scaling invariant $\gamma = \frac{b}{|\mu \bk|} \in \bbR \backslash \{ 0 \}$.  Since 
 $b = -\Re(\bi\eta(q,q))$ and $\mu = -\bj\eta(q,q) - b\bk$, the $S$-orbits in $\bbC\bbP^3 \backslash \cQ^\#$ satisfy
 \begin{align} \label{E:H-model1}
 \Re(\bi\eta(q,q)) = -\gamma |\bj\eta(q,q)\bk + \Re(\bi\eta(q,q))|.
 \end{align}
 Equivalently, $\Im(\overline{z_1} z_3 + \overline{z_2} z_4) = \gamma |\bi\Re(z_1 \overline{z_3} + z_2 \overline{z_4}) - (z_1 z_4 - z_2 z_3) \bk|$, which further simplifies to
 \begin{align} \label{E:H-model2}
 \Im(\overline{z_1} z_3 + \overline{z_2} z_4) 
 &= \gamma \sqrt{|\Re(z_1 \overline{z_3} + z_2 \overline{z_4})|^2 + |z_1 z_4 - z_2 z_3|^2}.
 \end{align}
 
 Let us find representatives for the $S$-orbits on $\bbC\bbP^3 \backslash \cQ^\#$.  Given $0 \neq q \in \bbH^2$, we may swap $q_1$ and $q_2$ if necessary to assume $q_1 \neq 0$, and then use $\SU(2)$ and a real rescaling to assume $q_1 = 1$.  Using $\begin{psmallmatrix} 1 & 0\\ -\Re(q_2) & 1\end{psmallmatrix} \in \SL(2,\bbR)$, we may assume $\Re(q_2) = 0$.  Using $e^{\bi\phi} \in \SU(2)$ and the right $\bbC$-action by $e^{-\bi\phi}$, we can replace $q_2 \mapsto e^{\bi \phi} q_2 e^{-\bi\phi}$, so choose $\phi$ to normalize $q_2 = \bi s + \bj t$, where $s,t \in \bbR$ with $t \leq 0$.  Finally, use $\diag(\lambda, \frac{1}{\lambda}) \in \SL(2,\bbR)$ and right multiplication by $\frac{1}{\lambda}$ to normalize $q_2$ to be of unit length.  Thus, we have $q_2 = \bi e^{\bk\theta}$, for some fixed $0 \leq \theta < \pi$.  For $q = \begin{psmallmatrix} 1\\ \bi e^{\bk\theta} \end{psmallmatrix}$, we have $\eta(q,q) = 2\bi e^{\bk\theta}$.  When $\theta = \frac{\pi}{2}$ or $0$, we have $b=0$ or $\mu = 0$, so we exclude these.  The ($S$-equivariant) conjugation arising from the identification $\bbH^2 = \bbC^4$, i.e.\ $\widetilde{u+\bj v} = \bar{u} + \bj \bar{v}$ (where $u,v \in \bbC^2$), maps $\begin{psmallmatrix} 1\\ \bi e^{\bk\theta} \end{psmallmatrix} \mapsto \begin{psmallmatrix} 1\\ \bi e^{\bk(\pi-\theta)} \end{psmallmatrix}$.  Thus, we can restrict to $\gamma = \cot(\theta) > 0$, or equivalently require $\theta \in (0,\frac{\pi}{2})$.  Hence, we can restrict to $0 < \theta < \frac{\pi}{2}$, which is in 1-1 correspondence with $\gamma = \cot\theta \in \bbR^+$.
 
 The elements $\bi,\bj,\bk$ span $\fsu(2)$, and let $\sfH = \begin{psmallmatrix} 1 &0\\ 0 & -1\end{psmallmatrix},\sfX=\begin{psmallmatrix} 0 &1\\ 0 & 0\end{psmallmatrix},\sfY=\begin{psmallmatrix} 0 &0\\ 1 & 0\end{psmallmatrix}$ span $\fsl(2,\bbR)$.  Consider the $S$-orbit $M$ through $[q] \in \bbC\bbP^3 \backslash \cQ^\#$, where $q = \begin{psmallmatrix} 1\\ \bi e^{\bk\theta} \end{psmallmatrix}$ for $0 < \theta < \frac{\pi}{2}$.
 Then $[q]$ has 1-dimensional isotropy $\fk$ spanned by $\sfT = -\bi e^{\bk\theta} + \sfX - \sfY$.  (Note $\fk = \tspan_\bbR\{ \bi, \sfX - \sfY \}$ when $\theta = 0$.)  The complex structure $J$ on $\bbC\bbP^3$ is induced from right multiplication by $\bi$ on $\bbH^2$.  Defining
 \begin{align}
  \sfw_1 &:= \bk, \quad
  \sfw_2 := \bj - 2\sin(\theta) \sfY, \quad
  \sfw_3 := \sfH, \quad
  \sfw_4 := \bi + 2\cos(\theta) \sfY,
  \end{align}
 the elements $\sfw_1 \cdot q,...,\sfw_4 \cdot q\, \mod [q]$ span the contact subspace $H_{[q]} \subset T_{[q]} M$.
 In this basis, $J = \diag(\begin{psmallmatrix} 0 & -1\\ 1 & 0\end{psmallmatrix}, \begin{psmallmatrix} 0 & -1\\ 1 & 0\end{psmallmatrix})$.  In $\tspan_\bbC\{ \sfw_1,..., \sfw_4, \sfT \}$, define
 \begin{align}
 \sfv_1 &:= i\sin(\theta) (\sfw_1 + i \sfw_2) - i(\cos(\theta)-1)(\sfw_3 + i \sfw_4) - (\cos(\theta)-1) \sfT,\\
 \sfv_2 &:= i\sin(\theta) (\sfw_1 + i \sfw_2) - i(\cos(\theta)+1)(\sfw_3 + i \sfw_4) + (\cos(\theta)+1) \sfT.
 \end{align}
 Here we have distinguished the scalar $i$ from the Lie algebra element $\bi$.  The $-i$-eigenspace $\fv / \fk$ for $J$ is spanned by $\{ \sfv_1, \sfv_2 \} \mod \sfT$, while the $+i$-eigenspace $\fe / \fk$ is spanned by $\{ \overline{\sfv_1},\overline{\sfv_2}\} \mod \sfT$.  (We caution that this conjugation fixes each of $\sfw_1,...,\sfw_4,\sfT$ and conjugates the scalars.  It is distinct from the conjugation on $\bbH$, and that associated with $\bbH^2 = \bbC^4$.)  The ILC D.6-3 structure equations are satisfied if we take
 \begin{align}
 \rho = \frac{\sqrt{3}}{4}\sqrt{1 + \sec(\theta)}, \quad
 \begin{cases}
 e_1 &= \frac{3(1+\cos(\theta))}{16\rho \cos(\theta)\sin(\theta)} \overline{\sfv_1}, \quad
 e_2 = \frac{i \rho}{1+\cos(\theta)}\overline{\sfv_2}, \\
 e_3 &= \frac{\rho}{\sin(\theta)} \sfv_1, \quad
 e_4 = \frac{3i}{16\rho\cos(\theta)} \sfv_2\\
 e_5 &= 
 \frac{3}{4} i\sec(\theta) (\bi e^{\bk\theta} + \sfX - \sfY), \quad
 e_6 = -\frac{i}{2} \sfT
 \end{cases}, \quad 
 a = 3i\tan\theta.
 \end{align}
 In particular, \framebox{$a^2 = -\frac{9}{\gamma^2} < 0$} classifies the corresponding ILC D.6-3 structure.
 
\section{Hypersurfaces of Winkelmann type}
\label{S:Winkelmann}

 Generalizing \eqref{E:Wink4}, we say that a real hypersurface in $\bbC^3$ is of \emph{Winkelmann type} if in some holomorphic coordinate system $(z_1,z_2,w)$, it is given by
\begin{equation}\label{eqWT}
\Im(w+\bar z_1 z_2) = F(z_1,\bar z_1),
\end{equation}
 where $F$ is an arbitrary real analytic function on $\bbC=\bbR^2$.  These all admit the symmetries
\begin{align*}
N_1 = z_1\partial_{z_2}, \qquad
N_2 =\partial_{z_2}+z_1\partial_{w}, \qquad
N_3 = i\partial_{z_2}-iz_1\partial_{w}, \qquad
N_4 = \partial_{w}.
\end{align*}
These span an {\em abelian} Lie algebra that induces a (holomorphic) foliation of $\bbC^3$ by 2-dimensional holomorphic hypersurfaces $z_1=\operatorname{const}$ outside of the singular set $z_1=0$.

Complexifying \eqref{eqWT}, we get the following complex hypersurfaces in $\bbC^3\times\bar\bbC^3$:
\begin{equation}\label{specILC}
w = b - a_1 z_2 + z_1 a_2 + 2iF(z_1,a_1).
\end{equation}
 Regarding $w = w(z_1,z_2)$, we obtain $w_1 := \frac{\partial w}{\partial z_1} = a_2 + 2iF_{z_1}(z_1,a_1)$ and $w_2 := \frac{\partial w}{\partial z_2} = -a_1$.  Differentiate with respect to $z_1$ and $z_2$ once more and eliminate the parameters $(a_1,a_2,b)$.  Making the variable change $z_2 \mapsto -z_2$, we arrive at the PDE system
\begin{align} \label{E:Wink-PDE}
w_{11} = 2iF_{z_1z_1}(z_1,w_2),\quad w_{12}=w_{22}=0.
\end{align}
The harmonic curvature \cite[(3.3)]{DMT2014} is of type N if $F_{z_1 z_1 w_2 w_2}\ne 0$.  Let us consider the specific Winkelmann type hypersurfaces given in Table \ref{F:Wink}.

 \begin{center}
 \begin{footnotesize}
 \begin{table}[h]
 \[
 \begin{array}{|c|c|c|} \hline
 \mbox{Model} & \mbox{ILC N.6-2 classification} & \mbox{CR symmetries aside from $N_1,N_2,N_3,N_4$}\\ \hline\hline
 \begin{array}{c}
 \Im(w + \bar{z}_1 z_2) = (z_1)^\alpha (\bar{z}_1)^{\bar\alpha}\\
 (\alpha \in \bbC \backslash \{ -1, 0, 1, 2\})
 \end{array} & \begin{array}{c} \begin{array}{c} \bar{b}^2 = a^2 = \frac{-(2\alpha-1)^2}{(\alpha+1)(\alpha-2)} \\
 \in \bbC \backslash \{ -4,\frac{1}{2} \} \end{array}\\
 \end{array}  &
 \begin{array}{c}
 z_1\partial_{z_1} + (\alpha + \bar\alpha - 1) z_2 \partial_{z_2} + (\alpha + \bar\alpha) w \partial_w,\\
 i z_1 \partial_{z_1} + i (\alpha - \bar\alpha +1) z_2 \partial_{z_2} + i(\alpha - \bar\alpha) w \partial_w
 \end{array}\\ \hline
 \Im(w + \bar{z}_1 z_2) = \exp(z_1)\exp(\bar{z}_1) & b^2 = a^2 = -4&
 \begin{array}{c}
 i\partial_{z_1} + iz_2\partial_w,\\
 \partial_{z_1} + 2z_2\partial_{z_2} + (2w - z_2)\partial_w
 \end{array}\\ \hline
 \Im(w + \bar{z}_1 z_2) = \ln(z_1) \ln(\bar{z}_1) & b^2 = a^2 = \frac{1}{2} & 
 \begin{array}{c}
 z_1\partial_{z_1} - z_2\partial_{z_2} + 2 i\ln(z_1)\partial_w,\\
 iz_1\partial_{z_1} + iz_2\partial_{z_2} + 2\ln(z_1)\partial_w
 \end{array}\\ \hline
 \end{array}
 \]
 \caption{CR structures underlying ILC N.6-2 models}
 \label{F:Wink}
 \end{table}
 \end{footnotesize}
 \end{center}

 Writing out \eqref{E:Wink-PDE} for $F(z_1,\bar{z}_1) = (z_1)^\alpha (\bar{z}_1)^{\bar\alpha}$, we obtain
 \begin{align} \label{E:Wink-PDE1}
  w_{11} = 2i \alpha(\alpha-1) (z_1)^{\alpha-2} (w_2)^{\bar\alpha}, \quad w_{12} = w_{22} = 0.
 \end{align}
 Using a constant rescaling of $z_2$, and relabelling, we can bring \eqref{E:Wink-PDE1} into the same form as that listed in \cite[Table 1.1]{DMT2014} for the N.6-2 models with $\mu = \bar\alpha$ and $\kappa = \alpha - 2$.  The $F(z_1,\bar{z}_1) = \exp(z_1)\exp(\bar{z}_1)$ and $F(z_1,\bar{z}_1) = \ln(z_1) \ln(\bar{z}_1)$ cases are handled similarly.  The corresponding ILC N.6-2 models have $\mu = \kappa = \infty$ for the former and $(\mu,\kappa) = (0,-2)$ for the latter, c.f.\ \cite[Table 1.1]{DMT2014}. 
 
  As discussed in \S \ref{S:RealForms}, the N.6-2 models are described using $(a,b) \in \bbC^2$, with $(a^2,b^2)$ being essential parameters.  Using \eqref{E:N62-params} with $\mu = \bar\alpha$ and $\kappa = \alpha - 2$, we find that for $\alpha \in \bbC \setminus \{ -1, 0, 1, 2 \}$:
 \[
 \bar{b}^2 = a^2 = \frac{-(2\alpha-1)^2}{(\alpha+1)(\alpha-2)} \in \bbC \setminus \left\{ -4, \frac{1}{2} \right\}.
 \]
 (The $\alpha = -1$ and $\alpha = 2$ cases lead to the ILC N.8 model.)  The $b^2 = a^2 \in \{ -4, \frac{1}{2} \}$ cases were described in Example \ref{EX:N62-tubular}.   From \S \ref{S:tubular}, the $b^2 = a^2 \in \bbR$ cases are tubular (see Table \ref{F:affineN} for models).
 
\begin{prop}\label{wink-domain}
 The hypersurfaces of Winkelmann type given in Table~\ref{F:Wink} are models for all real forms of the complex ILC N.6-2 structures.  Their (6-dimensional) CR symmetry algebras are never transitive outside of these hypersurfaces.
\end{prop}

\begin{proof}
 From Table \ref{F:AI-reps}, $b^2 = \bar{a}^2 \in \bbC$ are the parameter values that yield underlying CR structures, and in each case there is a {\em unique} structure.  Thus, Table \ref{F:Wink} gives a complete classification as claimed.

 For the second claim, let us fix a basepoint $o \in \bbC^3$ and suppose that $z_1|_o = c+di$, where $c,d \in \bbR$.  Then $N_1 - c N_2 - dN_3 + (c^2+d^2) N_4$ vanishes at $o$. Thus, these symmetry algebras have at most 5-dimensional orbits at all points in $\bbC^3$.
\end{proof}
 
\section{Transitivity of the symmetry algebra}
\label{sec:trans}

In this section, we prove Theorem \ref{T:trans}.  According to~\cite[Cor.6.36]{Merker2008}, $\Sym(M)$ is transitive on an open subset of $\bbC^3$ if and only if $\Sym(M^c)$ is transitive on an open subset of $\bbC^3\times \bar \bbC^3$.

\begin{prop} Suppose $M\subset \bbC^3$ is locally transitive, and let $(\fs,\fk;\fe,\fv)$ be an algebraic model of the ILC structure $(E,V)$ on $M^c$. Then $\Sym(M^c)$ is transitive on a non-empty open subset of $\bbC^3\times\bar \bbC^3$ if and only if for some $T\in \Int(\fs)$, we have:
 \begin{equation}\label{eq:intT}
	\fe + T(\fv) = \fs.
 \end{equation}
\end{prop}

\begin{proof}
Assume that $\fs=\Sym(M^c)$ is transitive on a non-empty open subset of $\bbC^3\times\bar \bbC^3$ and fix a point $(z,a)$ in this subset. Then the projection of $\fs$ on $\bbC^3$ is transitive on an open subset of $\bbC^3$ containing $z\in \bbC^3$. The isotropy subalgebra of this action at $z$ is conjugate to $\fv$ by means of some inner automorphism $T_1 \in \Int(\fs)$. Similarly, the projection of $\fs$ to $\bar\bbC^3$ is transitive at $a\in\bar\bbC^3$ with the isotropy subalgebra equal to $T_2(\fe)$ for some $T_2\in\Int(\fs)$.
	
Note that $\fs$ is transitive at $(z,a)$ if and only if the projection of $T_1(\fv)$ to $\bar \bbC^3$ is transitive at $a\in\bbC^3$, or, similarly, if the projection of $T_2(\fe)$ to $\bbC^3$ is transitive at $z\in\bbC^3$. Both these conditions are equivalent to the equality $T_1(\fv)+T_2(\fe)=\fs$.  Applying $T_2^{-1}$ to both sides of this equality we get $\fe+T(\fv)=\fs$, where $T=T_2^{-1}T_1$.
\end{proof} 

\begin{cor}\label{cor:compl}
The symmetry algebra $\fs = \Sym(M^c)$ is transitive on a non-empty open subset of $\bbC^3\times\bar \bbC^3$ if and only if the set of all $\epsilon\in \bbC$ satisfying: 
 \[
 \fe + \exp(\epsilon X)(\fv) = \fs
 \]	
 is non-empty for any $X\in\fs$ not contained in $\fe+\fv$.
\end{cor}

\begin{proof}
Fix $X \in \fs \backslash (\fe + \fv)$, and decompose $\fs$ into a direct sum of linear subspaces $V_1\oplus \bbC X \oplus V_2$, where $V_1\subset \fe$ and $V_2\subset \fv$. It is well-known that the exponential map:
 \[
 \exp\colon V_1\times \bbC \times V_2 \mapsto \Int(\fs),\quad (X_1,\epsilon, X_2)\mapsto \exp(X_1)\exp(\epsilon X)\exp(X_2)
 \]
 is locally biholomorphic in a neighborhood of $0$. Taking $T= \exp(X_1)\exp(\epsilon X)\exp(X_2)$ in \eqref{eq:intT} and multiplying both sides by $\exp(-X_1)$ we get:
 \[
 \exp(-X_1)(\fe) + \exp(\epsilon X)\exp(X_2)(\fv) = \fs
 \]
 But by construction $\exp(-X_1)(\fe)=\fe$ and $\exp(X_2)(\fv)=\fv$.
\end{proof}

\begin{cor}\label{cor:center}
 If $\fs$ has a non-trivial center, then $\fs$ is not transitive at any point of $\bbC^3\times \bar\bbC^3$.
\end{cor}

\begin{proof}
 Let $Z$ be any central element in $\fs$. Since the subalgebra $\fk$ is effective, then $Z\notin \fk$. Recall that Levi non-degeneracy of $M\subset \bbC^3$ implies that the bilinear form:
 \[
 \Lambda^2 (\fe+\fv)/\fk \to \fs /(\fe+\fv),\quad (X+\fk)\wedge (Y+\fk)\mapsto [X,Y]+(\fe+\fv)
 \] 
 is non-degenerate. Hence, it follows that $Z$ does not lie in $\fe+\fv$. But it is obvious that $\exp(\epsilon Z)(\fv)=\fv$ for any $\epsilon\in\bbC$, and $\fs$ is not transitive according to Corollary~\ref{cor:compl}.
\end{proof}

This proposition implies that the local transitivity of $\Sym(M^c)$ on $\bbC^3 \times \bar \bbC^3$, and hence the local transitivity of $\Sym(M)$ on $\bbC^3$, is a property of the algebraic model $(\fs,\fk;\fe,\fv)$ itself and does not depend on the particular realization of this model in local coordinates.

 Transitivity off the hypersurface is well-known in the maximally symmetric and Winkelmann hypersurface cases \cite{Wink1995}.  Consider all remaining cases.  For N.7-2, D.6-1, and D.7, $\fs$ has non-trivial center, so these are ruled out by Corollary \ref{cor:center}.  (See also Example~\ref{Ex:D7} where it is computed explicitly as a Lie algebra of vector fields on $\bbC^3 \times \bar\bbC^3$.)

 The remaining cases have 6-dimensional symmetry, so it suffices to exhibit a relation amongst the symmetry vector fields. For N.6-1, tubular realizations are given in Table \ref{F:affineN} and share the symmetries  $N_1=i\partial_{z_2}$, $N_2=i\partial_w$, $N_3=\partial_{z_2}+z_1\partial_w$, and $N_4=iz_1\partial_{z_2}+i\frac{z_1^2}{2}\partial_w$. Letting $z_1=a+bi$, we have $aN_1+\frac{a^2+b^2}{2}N_2-bN_3-N_4=0$.  The N.6-2 cases were similarly ruled out in  Proposition~\ref{wink-domain}.  The D.6-2 realizations (see Table \ref{F:affineD}) share the symmetries $N_1=i\partial_{z_2}$, $N_2=i\partial_w$, $N_3=\partial_{z_2}+2z_2\partial_w$, and $N_4=iz_2\partial_{z_2}+iz_2^2\partial_w$. Letting $z_2=c+di$, we have $cN_1+(c^2+d^2)N_2-dN_3-N_4=0$.  For D.6-3 ($a^2 = 9$), the symmetries of the tubular models in Table \ref{F:affineD} have an obvious dependency.  Finally, for D.6-3 ($a^2 \neq 9$), see \S \ref{S:Cartan} where the orbits of $\Sym(M)$ are described explicitly.
 
 \section*{Acknowledgements}
 
 The authors gratefully acknowledge the use of the {\tt DifferentialGeometry} package in {\tt Maple}.  D.T. was supported by the University of Troms\o{} and project M1884-N35 of the Austrian Science Fund (FWF).

\newpage
\appendix
  
 \section{Representative admissible anti-involutions}
 \label{A:AI}
 
In Table \ref{F:AI-reps}, we classify all representative admissible anti-involutions (see \S \ref{S:RealForms}) for all non-flat 5-dimensional multiply-transitive complex ILC structures.  Each anti-involution is expressed in the same basis used in \cite[Tables 4.2--4.4]{DMT2014}.  Below, let $\epsilon = \pm 1$.
 
 \begin{center}
 \begin{tiny}
 \begin{table}[h]
 \[
 \begin{array}{|l|c|c|c|c|} \hline
 \mbox{Model}  & \mbox{Representative anti-involution} & 
 \mbox{Parameter conditions} & \mbox{Levi form type}\\ \hline\hline
 \mbox{N.8} & 
 \varphi = \diag\left( \begin{bmatrix}
     0 & 0 & 0 & 1 \\
     0 & 0 & 1 & 0 \\
     0 & 1 & 0 & 0 \\
     1 & 0 & 0 & 0
     \end{bmatrix},-1,-1,
     \begin{bmatrix} 0 & 1 \\ 1 & 0 \end{bmatrix}  \right)
 & - & \mbox{indefinite}
 \\
 \mbox{N.7-2} & \varphi^{(\epsilon)} = 
 \diag\left(
  \begin{bmatrix} 
  0 & 0 & 0 & \epsilon \\
  0 & 0 & \epsilon & 0 \\
  0 & \epsilon & 0 & 0 \\
  \epsilon & 0 & 0 & 0
  \end{bmatrix}, -1, -1, -1\right)
 & - & \mbox{indefinite}
 \\
 \mbox{N.6-1} &  
 \varphi = \diag\left( \begin{bmatrix}
0 & 0 & 0 & \tau \\
0 & 0 & \tau & 0 \\
0 & \frac{1}{\tau} & 0 & 0 \\
 \frac{1}{\tau} & 0 & 0 & 0\\
\end{bmatrix} , -1 ,-1 \right)
  & \begin{array}{c} a^2 \in \bbR \backslash \{ -1,0 \} \\ \tau = \frac{a\bar{a}}{a^2+1} \end{array} & \mbox{indefinite}\\
 \mbox{N.6-2} & \varphi = \diag\left(
 \begin{bmatrix} 
 0 & 0 & 0 & \epsilon \\
 0 & 0 & \epsilon & 0 \\
 0 & \epsilon & 0 & 0 \\
 \epsilon & 0 & 0 & 0
 \end{bmatrix}, -1, -1\right) & \begin{array}{c} b = \epsilon \bar{a}, 
 \quad a \in \bbC\\ (\mbox{Set $\epsilon=1$ if $a =0$}) \end{array} & 
 \mbox{indefinite}
\\ \hline
 \mbox{D.7} & \varphi_1^{(\epsilon_1,\epsilon_2)} = \diag\left( \begin{bmatrix}
  0 & 0 & \epsilon_1 & 0 \\
  0 & 0 & 0 & \epsilon_2 \\
  \epsilon_1 & 0 & 0 & 0 \\
  0 & \epsilon_2 & 0 & 0
  \end{bmatrix}, -1,-1,-1\right)  &  \begin{array}{c} a\in \bbR; \\ \mbox{ if } a=0 \mbox{ then } \\ (\epsilon_1,\epsilon_2)\in \{\pm (1,1), (1,-1) \} \end{array}& 
  \mbox{definite if } 
  \epsilon_1\epsilon_2=1
  \\
  & \varphi_2 = \diag\left( \begin{bmatrix}
    0 & 0 & 0 & 1 \\
    0 & 0 & 1 & 0 \\
    0 & 1 & 0 & 0 \\
    1 & 0 & 0 & 0
    \end{bmatrix},-1,-1,1 \right)
  & a \in i\bbR &  \mbox{indefinite} \\
 \mbox{D.6-1} 
 & \varphi^{(\epsilon)} = \diag\left( \begin{bmatrix}
   0 & 0 & 1 & 0 \\
   0 & 0 & 0 & \epsilon \\
   1 & 0 & 0 & 0 \\
   0 & \epsilon & 0 & 0
   \end{bmatrix},-1,-1 \right)  & - & \mbox{definite if } 
   \epsilon=1\\
  
  \mbox{D.6-2} & \varphi^{(\epsilon)} = 
  \diag\left( \begin{bmatrix}
     0 & 0 & \epsilon & 0 \\
     0 & 0 & 0 & \frac1\tau \\
     \epsilon & 0 & 0 & 0 \\
     0 & \tau & 0 & 0
     \end{bmatrix},-1,-1 \right) 
   & \begin{array}{c} a \in \bbR \backslash \{ 1, \frac{2}{3} \} \\ 
   \tau=-\frac{(3a-2)(a-1)}9 \end{array}
  & \mbox{definite if } \epsilon\tau>0\\
 \mbox{D.6-3} &  \varphi_1=
 \diag\left( \begin{bmatrix}
    0 & 0 & 0 & 1 \\
    0 & 0 & 1 & 0 \\
    0 & 1 & 0 & 0 \\
    1 & 0 & 0 & 0
    \end{bmatrix},-1,1 \right)

 & a\in\bbR \backslash \{0\}  & \mbox{indefinite}
 \\ 
  & \varphi_2^{(\epsilon)} = \diag\left( \begin{bmatrix}
   0 & 0 & \epsilon & 0 \\
   0 & 0 & 0 & \epsilon \\
   \epsilon & 0 & 0 & 0 \\
   0 & \epsilon & 0 & 0
   \end{bmatrix},-1,-1 \right) &  
 a\in \bbR \backslash \{0\}& \mbox{definite}
 \\
 & \varphi_3 = \diag\left( \begin{bmatrix}
    0 & 0 & 1 & 0 \\
    0 & 0 & 0 & -1 \\
    1 & 0 & 0 & 0 \\
    0 & -1 & 0 & 0
    \end{bmatrix},-1,-1 \right)   & a\in i \bbR \backslash \{0\}& 
    \mbox{indefinite}\\ \hline
 \end{array}
 \]

 \caption{Representative admissible anti-involutions}
 \label{F:AI-reps}
 \end{table}
 \end{tiny}
 \end{center}
 
 \begin{framed}
 There are ILC parameter redundancies resulting from certain basis changes \cite[Table A.5]{DMT2014}:
 \begin{itemize}
 \item N.6-1, D.7, and D.6-3: $a \mapsto -a$.
 \item N.6-2: $a \mapsto -a$ and $b \mapsto -b$ are independent redundancies.
 \item D.6-2: none.
 \end{itemize}
 These have no effect on anti-involutions except in the D.7 case: $(\varphi_1^{(\epsilon_1,\epsilon_2)},\varphi_2) \mapsto (\varphi_1^{(\epsilon_2,\epsilon_1)}, \varphi_2)$.
 \end{framed}
 
 \newpage 
 
 \section{Tubular hypersurfaces}
 \label{A:tubular}
 
 In Tables \ref{F:affineN} and \ref{F:affineD}, we give the complete (local) classification of (non-flat) homogeneous tubular hypersurfaces in $\bbC^3$ with non-degenerate Levi form, organized according to Petrov type.  The third column classification is given in terms of our ILC classification and the anti-involutions presented in Table \ref{F:AI-reps}.  (No anti-involution is specified if there is a unique one.)  We let $\epsilon = \pm 1$ here, and write $x = \Re(z_1), y = \Re(z_2)$, and $u = \Re(w)$.  CR parameter redundancies are indicated, e.g. $\alpha \sim -\alpha$.
 
 By Theorem \ref{s4:t1}, the structures excluded from this list are D.6-3 $(a^2 \in \bbR \backslash \{ 0, 9 \})$ and N.6-2 ($b^2 = \overline{a}^2 \in \bbC \backslash \bbR$).  These are discussed in \S \ref{S:Cartan} and \S \ref{S:Winkelmann} respectively.

 \begin{Tiny}
 \begin{table}[h]
 \[
 \begin{array}{|c@{}|@{}c@{}|@{}c@{}|c|} \hline
 \begin{array}{c}
 \mbox{Real affine surface} \\
 F(x,y,u) = 0
 \end{array}
 & 
 \begin{tabular}{c}
 Affine \\
 hom.?
 \end{tabular} 
 &
 \begin{tabular}{c}
 Classification
 \end{tabular}
 & 
 \begin{tabular}{c}
 CR syms of $F(\Re(z_1),\Re(z_2),\Re(w)) = 0$\\
 beyond $i\partial_{z_1}, i\partial_{z_2}, i\partial_{w}$
 \end{tabular}
 \\ \hline\hline
 %%%%%%%%
 u = xy + x^4 & \checkmark & \mbox{N.8} &
 \begin{array}{c}
 \partial_{z_2} + z_1\partial_w,\\
 i z_1\partial_{z_2} + i \frac{(z_1)^2}{2} \partial_w,\\
 z_1\partial_{z_1} + 3z_2\partial_{z_2} + 4w\partial_w,\\
 \partial_{z_1} - 6 (z_1)^2 \partial_{z_2} + (z_2 - 2(z_1)^3) \partial_w,\\
 i z_1\partial_{z_1} + i (z_2-2(z_1)^3) \partial_{z_2} + i (z_1 z_2 - \frac{(z_1)^4}{2}) \partial_w
 \end{array} \\ \hline
 %%%%%%%%
 u = xy + x\ln(x) & \checkmark & \begin{array}{c} \mbox{N.7-2}\\ \fsl(2,\bbR)\ltimes (V_2\oplus V_0) \\ \varphi^{(+1)} \end{array}& 
 \begin{array}{c}
 \partial_{z_2} + z_1\partial_w,\\
 i z_1\partial_{z_2} + i \frac{(z_1)^2}{2}\partial_w,\\
 z_1\partial_{z_1} - \partial_{z_2} + w\partial_w,\\
 i \frac{(z_1)^2}{2}\partial_{z_1} + i (w-z_1)\partial_{z_2} + i w z_1\partial_w 
 \end{array} \\ \hline
  %%%%%%%%
 \begin{array}{c} u = Xy + X\ln(X), \\ X=\exp(2x)+1 \end{array} & \times & \begin{array}{c} \mbox{N.7-2} \\ \fsu(2)\ltimes (V_2\oplus V_0) \\ \varphi^{(-1)} \end{array} & 
  \begin{array}{c}
\cosh(z_1)\partial_{z_1}-\left(\frac{1}{2}\exp(-z_1)w + \exp(z_1) \right)\partial_{z_2}+w\sinh(z_1)\partial_w,
\\
\exp(-z_1)\partial_{z_2}+2\cosh(z_1)\partial_w,\\
i\left( \exp(-z_1)\partial_{z_2}-2\sinh(z_1)\partial_w \right),
\\
i\sinh(z_1)\partial_{z_1}+i\left(\frac{1}{2}\exp(-z_1)w-\exp(z_1)\right)\partial_{z_2}+iw\cosh(z_1)\partial_w 
  \end{array} \\ \hline
 %%%%%%%%
 \begin{array}{c} u = xy + x^\alpha \\ (\alpha \in \bbR \backslash \{ 0, 1, 2, 3, 4 \}) \end{array} & \checkmark &
 \begin{array}{c} \mbox{N.6-1} \\ a^2 = \frac{1-\alpha}{\alpha-4}\\ \in \bbR \backslash \{ -1, -\frac{1}{4}, 0, \frac{1}{2}, 2 \} \end{array} &
 \begin{array}{c}
 \partial_{z_2} + z_1\partial_w,\\
 i z_1\partial_{z_2} + i \frac{(z_1)^2}{2}\partial_w,\\
 z_1\partial_{z_1} + (\alpha-1)z_2\partial_{z_2} + \alpha w\partial_w
 \end{array}\\ \hline
 %%%%%%%%
 u = xy + \ln(x) & \checkmark & 
 \begin{array}{c} \mbox{N.6-1}\\ a^2 = -\frac{1}{4} \end{array} &
 \begin{array}{c}
 \partial_{z_2} + z_1\partial_w, \\ 
 i z_1\partial_{z_2} + i\frac{(z_1)^2}{2}\partial_w,\\
 z_1\partial_{z_1} - z_2\partial_{z_2} + \partial_w
 \end{array}\\ \hline
 %%%%%%%%
 u = xy + x^2\ln(x) & \checkmark & 
 \begin{array}{c} \mbox{N.6-1}\\ a^2 = \frac{1}{2} \end{array} &
 \begin{array}{c}
 \partial_{z_2} + z_1\partial_w, \\ 
 i z_1\partial_{z_2} + i\frac{(z_1)^2}{2}\partial_w,\\
 z_1\partial_{z_1} + (z_2-z_1)\partial_{z_2} + 2w \partial_w
 \end{array}\\ \hline
 %%%%%%%%
 u = xy + x^3 \ln(x) & \times & 
 \begin{array}{c} \mbox{N.6-1}\\ a^2 = 2 \end{array} &
  \begin{array}{c}
 \partial_{z_2} + z_1\partial_w,\\ 
 i z_1\partial_{z_2} + i\frac{(z_1)^2}{2}\partial_w,\\
 z_1\partial_{z_1} + (2z_2-\frac{3}{2} (z_1)^2)\partial_{z_2} + (3w-\frac{1}{2}(z_1)^3) \partial_w
 \end{array}\\ \hline
 %%%%%%%%
 \begin{array}{c} u = y\exp(x) + \exp(\alpha x) \\ (\alpha \in \bbR \backslash \{ -1,0,1,2 \}; \\ 
 \alpha \sim 1-\alpha)\end{array} &
 \times
 &
 \begin{array}{c} \mbox{N.6-2} \\ b^2 = a^2 = \frac{-(2\alpha-1)^2}{(\alpha+1)(\alpha-2)} \\ \in \bbR \backslash ( [-4,0) \cup \{ \frac{1}{2} \}) \end{array} &
 \begin{array}{c}
 \partial_{z_1} + (\alpha-1) z_2 \partial_{z_2} + \alpha w \partial_w,\\
 \exp(\frac12 z_1) \partial_w + \exp(-\frac12 z_1) \partial_{z_2},\\
 i \exp(\frac12 z_1) \partial_w - i \exp(-\frac12 z_1) \partial_{z_2}
 \end{array}\\ \hline
%%%%%%%%
 \begin{array}{c} u\cos(x)+y\sin(x) = \exp(\beta x) \\ (\beta \in \bbR; \, \beta \sim -\beta ) \end{array} &
 \times
 &
 \begin{array}{c} \mbox{N.6-2} \\ b^2 = a^2 = \frac{-4\beta^2}{\beta^2+9} \in  (-4,0] \end{array} &
 \begin{array}{c}
 \partial_{z_1}-(\beta z_2+w)\partial_{z_2} +(z_2-\beta w)\partial_{w}, \\
 \sin(z_1)\partial_{z_2}-\cos(z_1)\partial_w, \\
 -i\cos(z_1)\partial_{z_2}-i\sin(z_1)\partial_w
 \end{array}\\ \hline
 %%%%%%%%
 \begin{array}{c} u = xy + \exp(x) \end{array} & \checkmark  &
 \begin{array}{c} \mbox{N.6-2} \\ b^2 = a^2 = -4 \end{array} &
 \begin{array}{c}
 \partial_{z_2} + z_1\partial_w,\\
 iz_1 \partial_{z_2} + i \frac{(z_1)^2}{2} \partial_w,\\
 \partial_{z_1} + z_2 \partial_{z_2} + (w+z_2)\partial_w
 \end{array}\\ \hline 
 %%%%%%%%
 \begin{array}{c} u = y\exp(x) - \frac{x^2}{2} \end{array} & \times &
 \begin{array}{c} \mbox{N.6-2} \\ b^2 = a^2 = \frac{1}{2} \end{array} &
 \begin{array}{c}
 \partial_{z_1} - z_2 \partial_{z_2} - z_1 \partial_w,\\
 \exp(\frac12 z_1) \partial_w + \exp(-\frac12 z_1) \partial_{z_2},\\
 i\exp(\frac12 z_1) \partial_w - i\exp(-\frac12 z_1) \partial_{z_2}
 \end{array}\\ \hline 
 \end{array}
 \]
 \caption{Real affine surfaces and symmetries of corresponding tubular CR structures: type N cases}
 \label{F:affineN}
 \end{table}
 \end{Tiny} 
 
 \begin{remark} For $u = y\exp(x) + \exp(\alpha x)$ in the N.6-2 case, $\alpha=0,1$ lead to the flat model, while $\alpha=-1,2$ give alternative descriptions of the N.8 model.
 \end{remark}

\newpage 

\begin{Tiny}
 \begin{table}[h]
 \[
 \begin{array}{|c@{}|@{}c@{}|@{}c@{}|c|} \hline
 \begin{array}{c}
 \mbox{Real affine surface} \\
 F(x,y,u) = 0
 \end{array}
 & 
 \begin{tabular}{c}
 Affine \\
 hom.?
 \end{tabular} 
 &
 \begin{tabular}{c}
 Classification
 \end{tabular}
 & 
 \begin{tabular}{c}
 CR syms of $F(\Re(z_1),\Re(z_2),\Re(w)) = 0$\\
 beyond $i\partial_{z_1}, i\partial_{z_2}, i\partial_{w}$
 \end{tabular}
 \\ \hline\hline
 %%%%%%%%
 \begin{array}{c} u = \alpha\ln(x) + \ln(y) \\
 (\alpha \in \bbR \backslash \{ -1,0 \}; \, \alpha \sim \frac{1}{\alpha} )
  \end{array} 
 & \checkmark &
\begin{array}{c} \mbox{D.7} \\
 \fsl(2,\bbR)\times \fsl(2,\bbR) \times \bbR\\ a=\frac{3}{4}(\frac{\alpha-1}{\alpha+1}) \in \bbR \backslash \{ \pm \frac{3}{4} \}\\ 
 \varphi_1^{(-1,-1)},\, |a| < \frac{3}{4};\\
 \varphi_1^{(1,-1)}, \, a > \frac{3}{4};\\
 \varphi_1^{(-1,1)}, \, a < -\frac{3}{4}
 \end{array} &
 \begin{array}{cc}
 z_1\partial_{z_1} + \alpha\partial_w, \\
 z_2\partial_{z_2} + \partial_w, \\ 
 i(z_1)^2 \partial_{z_1} + 2i \alpha z_1 \partial_w, \\
 i(z_2)^2 \partial_{z_2} + 2 iz_2 \partial_w\\ 
 \end{array}\\ \hline
 %%%%%%%%
 \begin{array}{c} 
 u = \alpha\ln(X) + \ln(y), \\ X=\exp(2x)+1 \\
 (\alpha \in \bbR \backslash \{ -1,0 \}; \, \alpha \sim \frac{1}{\alpha} )
 \end{array} 
 & \times &
 \begin{array}{c} \mbox{D.7}\\
 \fsl(2,\bbR)\times \fsu(2) \times \bbR\\
 a=\frac{3}{4}(\frac{\alpha-1}{\alpha+1}) \in \bbR \backslash \{ \pm \frac{3}{4} \}\\
 \varphi_1^{(1,1)}, a < -\frac34; \\
 \varphi_1^{(1,-1)}, \, |a| < \frac{3}{4};\\
 \varphi_1^{(-1,-1)}, a > \frac34
 \end{array} &
 \begin{array}{cc}
 z_2\partial_{z_2} + \partial_w, \\ 
 i(z_2)^2 \partial_{z_2} + 2 iz_2 \partial_w,\\ 
 \cosh(z_1)\partial_{z_1} + \alpha \exp(z_1) \partial_w, \\
 i\sinh(z_1) \partial_{z_1} + i \alpha \exp(z_1) \partial_w
 \end{array}\\ \hline
 %%%%%%%%
 \begin{array}{c}
 u = \alpha\ln(X) + \ln(Y), \\
  X=\exp(2x)+1,\\ Y=\exp(2y)+1 \\
 (\alpha \in \bbR \backslash \{ -1,0 \}; \, \alpha \sim \frac{1}{\alpha})
 \end{array} & \times &
 \begin{array}{c} \mbox{D.7}\\
 \fsu(2)\times \fsu(2) \times \bbR \\
 a=\frac{3}{4}(\frac{\alpha-1}{\alpha+1}) \in \bbR \backslash \{ \pm \frac{3}{4} \}\\
 \varphi_1^{(1,1)},\, |a| < \frac{3}{4};\\
 \varphi_1^{(-1,1)}, \, a > \frac{3}{4};\\
 \varphi_1^{(1,-1)}, \, a < -\frac{3}{4}
 \end{array}&
 \begin{array}{cc}
 \cosh(z_1)\partial_{z_1} + \alpha \exp(z_1) \partial_w, \\
 \cosh (z_2)\partial_{z_2} + \exp(z_2) \partial_w, \\ 
 i\sinh(z_1) \partial_{z_1} + i \alpha \exp(z_1) \partial_w, \\
 i\sinh(z_2) \partial_{z_2} + i \exp(z_2) \partial_w\\
 \end{array}\\ \hline
 %%%%%%%%
 \begin{array}{c}
 u = \alpha \arg(ix+y) + \ln(x^2+y^2) \\
 (\alpha \in \bbR; \, \alpha \sim -\alpha)
 \end{array} & \checkmark &
 \begin{array}{c} \mbox{D.7} \\ 
 \fsl(2,\bbC)_\bbR \times \bbR\\
 \varphi_2,\, a=\frac{3}{8} i \alpha
 \end{array}
 &
 \begin{array}{c}
 z_1\partial_{z_1}+ z_2\partial_{z_2}+ \partial_{w},\\
 z_2\partial_{z_1}- z_1\partial_{z_2}+ \alpha\partial_{w},\\
 \frac{z_1^2-z_2^2}2\partial_{z_1}+ z_1 z_2\partial_{z_2}+ ( z_1-\alpha z_2)\partial_{w},\\
 z_1 z_2 \partial_{z_1}-\frac{z_1^2-z_2^2}2\partial_{z_2}+ (\alpha z_1+ z_2)\partial_{w}
 \end{array}
 \\ \hline
 %%%%%%%%
 u = y^2 + \epsilon \ln(x) & \checkmark &
 \begin{array}{c} \mbox{D.7} \\ \mbox{Semisimple part = } \fsl(2,\bbR) \\
 \varphi^{(\epsilon,-1)}, \,a = \frac{3}{4} \end{array} &
 \begin{array}{c}
 \partial_{z_2} + 2 z_2 \partial_w, \\
 z_1\partial_{z_1} + \epsilon \partial_w,\\
 i z_2 \partial_{z_2} + i (z_2)^2 \partial_w,\\
 i(z_1)^2 \partial_{z_1} + 2i \epsilon z_1 \partial_w
 \end{array}\\ \hline
 %%%%%%%%
  \begin{array}{c}
  u = y^2 + \epsilon \ln(X), \\ X=\exp(2x)+1
  \end{array} & \times &
  \begin{array}{c} \mbox{D.7} \\ \mbox{Semisimple part = } \fsu(2)\\
  \varphi^{(\epsilon,1)}, \,a = \frac{3}{4} \end{array} &
  \begin{array}{c}
  \partial_{z_2} + 2 z_2 \partial_w,\\
  i z_2 \partial_{z_2} + i (z_2)^2 \partial_w,\\
  \cosh(z_1)\partial_{z_1} + \epsilon \exp(z_1) \partial_w,\\
  i\sinh(z_1) \partial_{z_1} + i \epsilon \exp(z_1) \partial_w
  \end{array}\\ \hline
 %%%%%%%%
 xu = y^2 - \epsilon x\ln(x) & \checkmark &
 \begin{array}{c} \mbox{D.6.1} \\ \varphi^{(\epsilon)} \end{array} &
 \begin{array}{c}
 z_1 \partial_{z_2} + 2 z_2 \partial_w,\\
 2z_1 \partial_{z_1} + z_2 \partial_{z_2} - 2\epsilon \partial_w,\\
 i (z_1)^2 \partial_{z_1} + i z_1 z_2 \partial_{z_2} + i( -2\epsilon z_1 + (z_2)^2) \partial_w
 \end{array}\\ \hline
 %%%%%%%%
 \begin{array}{c} u = y^2 + \epsilon x^\alpha \quad (x > 0) \\ (\alpha \in \bbR \backslash \{ 0, 1, 2 \}) \end{array} & \checkmark & 
 \begin{array}{c} \mbox{D.6-2} \\ 
 a = \frac{2}{3}(\frac{\alpha+1}{\alpha}) \in \bbR \backslash \{ \frac{2}{3}, \frac{4}{3}, 1 \}\\
 \varphi^{(\rho)},\, \rho = \epsilon\,\sgn[ (a-\frac{2}{3})(a-\frac{4}{3})(a-1)]\end{array} &
 \begin{array}{c}
 \partial_{z_2} + 2 z_2\partial_w,\\ 
 i z_2\partial_{z_2} + i (z_2)^2\partial_w,\\
 z_1\partial_{z_1} + \frac{\alpha z_2}{2} \partial_{z_2} + \alpha w\partial_w
 \end{array}\\ \hline
 %%%%%%%%
 u = y^2 + \epsilon x\ln(x) & \checkmark &
 \begin{array}{c} \mbox{D.6-2} \\ \varphi^{(-\epsilon)}, a = \frac{4}{3} \end{array} & 
 \begin{array}{c}
 \partial_{z_2} + 2z_2\partial_w,\\ 
 i z_2\partial_{z_2} + i(z_2)^2\partial_w,\\
 z_1\partial_{z_1} + \frac{1}{2} z_2 \partial_{z_2} + (\epsilon z_1 + w)\partial_w
 \end{array}\\ \hline
 %%%%%%%%
 \begin{array}{c}
  u^2 + \epsilon_1 x^2 + \epsilon_2 y^2 = 1 \\
  (\epsilon_1,\epsilon_2) \in \{ \pm (1,1),(1,-1) \}
 \end{array} & \checkmark & \begin{array}{c} \mbox{D.6-3}\\ 
 a^2 = 9\\ 
 \begin{array}{@{}l@{\,}l@{\,}l@{}}
 \fso(1,2) \ltimes \bbR^3, &
 \varphi_1, & (\epsilon_1,\epsilon_2) = (+1,-1);\\
 \fso(3) \ltimes \bbR^3, & \varphi_2^{(+1)}, & (\epsilon_1,\epsilon_2) = (+1,+1);\\
 \fso(1,2) \ltimes \bbR^3, & \varphi_2^{(-1)}, & (\epsilon_1,\epsilon_2) = (-1,-1)
 \end{array}
 \end{array} &
 \begin{array}{c}
 \epsilon_1 z_2\partial_{z_1} - \epsilon_2 z_1\partial_{z_2},\\
 w\partial_{z_1} - \epsilon_1 z_1\partial_w,\\
 w\partial_{z_2} - \epsilon_2 z_2\partial_w
 \end{array}\\ \hline 
 \end{array}
 \]
 \caption{Real affine surfaces and symmetries of corresponding tubular CR structures: type D cases}
 \label{F:affineD}
 \end{table}
 \end{Tiny} 
 
 \begin{remark} 
 In the first three D.7 cases, $\alpha = -1$ yields the flat model.  This is also true in the D.6-2 case when $\alpha = 0,1,2$.
 \end{remark}

 \section{Loboda's models}
 \label{A:Loboda}
 
 In Table \ref{F:Loboda}, we give a dictionary between Loboda's classifications and our results. The first two series of examples describe all non-degenerate hypersurfaces with 7-dimensional symmetry algebra and indefinite~\cite{Loboda2001a} or definite~\cite{Loboda2001b} Levi form. 
 The last seven rows correspond to hypersurfaces with 6-dimensional symmetry algebra and positive-definite Levi form found in~\cite{Loboda2003}.\footnote{As noted in the Introduction, a D.6-1 real form is missing from Loboda's list.}  All equations here use the notation $z_j=x_j+i y_j$ and $w=u+iv$. 
 
 Loboda's models are not in the tubular form \eqref{s4:tube}, but we find the corresponding complex ILC structures as in \S \ref{S:Complex} (by replacing barred variables with parameters), see for instance Example \ref{Ex:Complex}.  We then proceed similarly as in the examples in \S \ref{S:tubular} to identify these models in our classification.

 \begin{tiny}
 \begin{table}[h]
 \[
 \begin{array}{|l|l|l|l|} \hline
 \mbox{Number} &
 \mbox{Surface} &
 \mbox{Parameter} &
 \mbox{Our classification}\\ 
 \hline\hline
 \mbox{7D Indef (2) } 
 &
 v=(z_1\bar z_2+z_2\bar z_1) + (1+\varepsilon |z_1|^2)\ln(1+\varepsilon |z_1|^2)
 & 
 \varepsilon=\pm1
 & \mbox{N.7-2, }\varphi^{(-\varepsilon)}
 \\
 \mbox{7D Indef (3)} 
 &
 v=e^{i\theta}\ln(1+z_1\bar z_2) + e^{-i\theta}\ln(1+z_2\bar z_1)
 & 
 \theta \in (-\frac\pi2,\frac\pi2)
 &
 \mbox{D.7, } \varphi_2,\,\, a= \frac{3i}4 \tan(\theta)
 \\
 \mbox{7D Indef (4)}
 &
 v=\ln(1-|z_1|^2)-b\ln(1-|z_2|^2)
 &
 b \in (0,1)
 &
 \mbox{D.7, } \varphi_1^{(1,-1)},\,\, a =  \frac34 \frac{1+b}{1-b}
 \\
 \mbox{7D Indef (5)} &
 v=\ln(1+|z_1|^2)+b\ln(1-|z_2|^2)
 &
 b \in (0,\infty)
 &
 \mbox{D.7, } \varphi_1^{(1,-1)},\,\, a = \frac34 \frac{1-b}{1+b}
 \\
 \mbox{7D Indef (6)} &
 v=\ln(1+|z_1|^2)-b\ln(1+|z_2|^2)
 &
 b \in (0,1)
 &
 \mbox{D.7, } \varphi_1^{(-1,1)},\,\, a = \frac34 \frac{1+b}{1-b}
 \\
 \mbox{7D Indef (7)} &
 v=|z_2|^2+\varepsilon\ln(1-\varepsilon |z_1|^2)
 &
 \varepsilon=\pm1
 &
 \mbox{D.7, } \varphi_1^{(\varepsilon,-\varepsilon)},\, a = \frac{3}{4}
 \\ \hline
 \mbox{7D Def (0.1)} &
 v=\ln(1+|z_1|^2)+b\ln(1+|z_2|^2)
 &
 b \in (0,1]
 &
 \mbox{D.7, }\varphi_1^{(1,1)},\,\, a = \frac34 \frac{1-b}{1+b} 
 \\
 \mbox{7D Def (0.2)} &
 v=\ln(1+|z_1|^2)-b\ln(1-|z_2|^2)
 &
 b \in (0,1)
 &
  \mbox{D.7, } \varphi_1^{(-1,-1)},\,\, a =  \frac34 \frac{1+b}{1-b}
 \\
  \mbox{7D Def (0.2)} &
  v=\ln(1+|z_1|^2)-b\ln(1-|z_2|^2)
  &
  b \in (1,\infty)
  &
   \mbox{D.7, } \varphi_1^{(1,1)},\,\, a =  \frac34 \frac{1+b}{1-b}
  \\
 \mbox{7D Def (0.3)} &
 v=\ln(1-|z_1|^2)+b\ln(1-|z_2|^2)
 &
 b \in (0,1]
 &
 \mbox{D.7, } \varphi_1^{(-1,-1)},\,\, a = \frac34 \frac{1-b}{1+b}
 \\
 \mbox{7D Def (0.4)} &
 v=|z_2|^2+\varepsilon\ln(1+\varepsilon|z_1|^2)
 &
 \varepsilon=\pm1
 &
 \mbox{D.7, } \varphi_1^{(\varepsilon,\varepsilon)},\,\, a=\frac34
 \\ \hline
 \mbox{6D Def (1)} &
 v=x_2^2+(1+x_1)^\alpha - 1 &
 \alpha\in (-\infty,0)\cup(1,2) &
 \mbox{D.6-2, } \varphi^{(-1)},\,\, a=\frac23\frac{\alpha+1}\alpha
 \\
\mbox{6D Def (1)} &
 v=x_2^2+(1+x_1)^\alpha - 1 &
 \alpha\in (2,\infty) &
 \mbox{D.6-2, } \varphi^{(1)},\,\, a=\frac23\frac{\alpha+1}\alpha
 \\
 \mbox{6D Def (2)} &
 v=x_2^2-(1+x_1)^\alpha + 1 &
 \alpha\in (0,1) &
\mbox{D.6-2, } \varphi^{(-1)},\,\, a=\frac23\frac{\alpha+1}\alpha
 \\
 \mbox{6D Def (3)} &
 v=x_2^2+(1+x_1)\ln(1+x_1) &
 - &
 \mbox{D.6-2, } \varphi^{(-1)},\,\, a=\frac43
 \\
 \mbox{6D Def (4),(5)} &
 \epsilon(x_1^2+x_2^2)+u^2 = 1 &
 \epsilon = \pm1 &
 \mbox{D.6-3},\, a^2 = 9,\, \varphi_2^{(\epsilon)}
 \\
 \mbox{6D Def (6)} &
 1+ \epsilon(|z_1|^2+|z_2|^2)+|w|^2 = c|1+z_1^2+z_2^2+w^2| & c>1,\, \epsilon = \pm 1 &
 \mbox{D.6-3},\, a^2 = \frac{9}{c^2} < 9,\, \varphi_2^{(\epsilon)}
 \\
 \mbox{6D Def (7)} &
 1+\epsilon(|z_1|^2+|z_2|^2)-|w|^2=c|1+z_1^2+z_2^2-w^2| & 0<c<1,\, \epsilon = \pm 1 &
 \mbox{D.6-3},\, a^2 = \frac{9}{c^2} > 9,\, \varphi_2^{(-\epsilon)} \\ \hline
 \end{array}
 \]
 \caption{Correspondence with Loboda's models
 }
 \label{F:Loboda}
 \end{table}
 \end{tiny} 
 
  \section{Homogeneous 3-dimensional CR structures}
  \label{A:CR3}

 It is well-known that all Levi non-degenerate real hypersurfaces in $\bbC^2$ admit at most an 8-dimensional symmetry algebra.  Moreover, the submaximal symmetry dimension is 3 and \'E. Cartan gave a complete local classification of all such (homogeneous) models \cite[bottom of p.70]{Cartan1932a}.  Here we outline how this classification can be alternatively derived from the well-known classification of (complex) 2nd order ODE that are homogeneous (in fact, simply-transitive, so the isotropy subalgebra is everywhere trivial) under point symmetries \cite[Table 7]{Olver1995}.\footnote{The ODE $u'' = \frac{3p^2}{2u} + cu^3$ as listed in \cite[Table 7]{Olver1995} is flat when $c=0$ (so 8 symmetries) and all $c \neq 0$ are equivalent via scalings, so we normalized $c=1$. Similar normalizations were done in the other cases.}  Set $p = u'$ below.  All parameters that yield equivalent models are indicated, e.g.\ $\gamma \sim 3-\gamma$.
 
 \begin{tiny}
 \[
 \begin{array}{|c|c|l|l|}\hline
 \mbox{Label} & \mbox{Complex ODE $u'' = f(x,u,p)$} & \mbox{Point symmetries} & \mbox{Lie algebra structure}\\ \hline\hline
 %%%%%%%%%%%%%%%%%%%%%
 \mbox{(A)} & u'' = \frac{3p^2}{2u} + u^3 & 
 \begin{array}{l}
 e_1 = \partial_x, \quad
 e_2 = x \partial_x - u \partial_u - 2p\partial_p,\\
 e_3 = x^2 \partial_x - 2xu\partial_u - (4xp+2u)\partial_p
 \end{array} &
 \begin{array}{l}
 [e_1,e_2] = e_1\\{}
 [e_1,e_3] = 2e_2\\{}
 [e_2,e_3] = e_3
 \end{array} \\ \hline
 %%%%%%%%%%%%%%%%%%%%%
 \mbox{(B)} & \begin{array}{c} 
 u'' = 6 u p - 4 u^3+c (p-u^2)^{3/2}\\
 (c \in \bbC \backslash \{ 0 \}; \,\,\, c \sim -c)
 \end{array} &
 \begin{array}{l}
 e_1 = \partial_x, \quad
 e_2 = x \partial_x - u \partial_u - 2p\partial_p,\\
 e_3 = x^2 \partial_x - (2xu+1)\partial_u - (4xp+2u)\partial_p \\
 \quad\qquad + \frac{c}{2} e_2 - e_1
 \end{array} &
 \begin{array}{l}
 [e_1,e_2] = e_1\\{}
 [e_1,e_3] = \frac{c}{2} e_1 + 2e_2\\{}
 [e_2,e_3] =  2e_1 - \frac{c}{2} e_2 + e_3
 \end{array} \\ \hline
 %%%%%%%%%%%%%%%%%%%%%
 \mbox{(C)} & \begin{array}{c} u'' = p^\gamma \\ (\gamma \in \bbC \backslash \{ 0,1,2,3 \}; \,\,\,
 \gamma \sim 3-\gamma)
 \end{array} & 
 \begin{array}{l}
 e_1 = (\gamma-1) x \partial_x + (\gamma-2) u \partial_u - p\partial_p\\
 e_2 = \partial_x, \,\,\,
 e_3 = \partial_u, \\
 \end{array} & 
 \begin{array}{l}
 [e_1,e_2] = -(\gamma-1)e_2\\{}
 [e_1,e_3] = -(\gamma-2) e_3\\{}
 [e_2,e_3] = 0
 \end{array} \\ \hline
 %%%%%%%%%%%%%%%%%%%%%
 \mbox{(D)} & u'' = e^{-p} &
 \begin{array}{l}
 e_1 = \partial_x, \,\,\, e_2 = \partial_u,\\
 e_3 = x\partial_x + (x+u)\partial_u + \partial_p
 \end{array} & 
 \begin{array}{l}
 [e_1,e_2] = 0\\{}
 [e_1,e_3] = e_1+e_2\\{}
 [e_2,e_3] = e_2
 \end{array} \\ \hline
 \end{array}
 \]
 \end{tiny}
 
 For each model, pick a general point $o$, identify the (1-dimensional) subalgebras $\fe$ and $\fv$ of the point symmetry algebra $\fs$ corresponding to the line fields $E = \tspan\{ \partial_x + p\partial_u + f\partial_p\}$ and $V = \tspan\{ \partial_p \}$ at $o$, and then classify all anti-involutions of $\fs$ that swap $\fe$ and $\fv$.  (This is tedious, but straightforward.)  All representative such admissible anti-involutions are given below.

 \begin{tiny}
  \[
 \begin{array}{|c|c|c|c|c|} \hline
 \mbox{Label} & \mbox{General point} & \fe & \fv & \mbox{Anti-involutions swapping $\fe$ and $\fv$}\\ \hline\hline
 \mbox{(A)} & x=0,\,u=1,\,p=0 & e_1 - \frac{1}{2} e_3 & e_3 & \mbox{none}\\ \hline
 %%%%%%%%%%%%%%%%%%%%%
 \mbox{(B)} & x=0, \, u=0,\, p=1 & e_3 & e_2 & \varphi = \begin{pmatrix}
 -\epsilon & 0 & 0\\
 \frac{\epsilon}{4\sigma} (c\sigma - 2) & 0 & \frac{\epsilon}{\sigma}\\
 \frac{\epsilon}{4} (c\sigma - 2) & \sigma & 0\\
 \end{pmatrix}, \quad \epsilon = \frac{\bar{c}}{c} = \frac{\sigma}{\bar\sigma} = \pm 1, \quad \sigma^2 = \frac{4}{16 + c^2}\\ \hline
 %%%%%%%%%%%%%%%%%%%%%
 \mbox{(C)} & x=y=0,\, p=1 & e_2 + e_3 - e_1 & e_1 & 
 \begin{array}{ll} 
 \gamma \in \bbR \backslash \{ 0,1,2,3 \}: & 
 \varphi_1 = \begin{pmatrix} 1 & 0 & 0\\ -1 & -1 & 0\\ -1 & 0 & -1\end{pmatrix}\\
 \Re(\gamma) = \frac{3}{2}: &
 \varphi_2 = \begin{pmatrix} -1 & 0 & 0\\ 1 & 0 & 1 \\ 1 & 1 & 0\end{pmatrix}
 \end{array}\\ \hline
 %%%%%%%%%%%%%%%%%%%%%
 \mbox{(D)} & x=y=p=0 & e_1 + e_3 & e_3 & \varphi = 
 \begin{pmatrix} -1 & 0 & 1\\ 0 & -1 & 0\\ 0 & 0 & 1\end{pmatrix}\\ \hline
 \end{array}
 \]
 \end{tiny}
 
 It is easy to recognize tubular CR structures in this list.  These arise from (C) and (D), since each admits a unique abelian subalgebra $\fa$ (namely, $\tspan\{ e_2, e_3 \}$ and $\tspan\{ e_1, e_2 \}$ respectively) that is complementary to both $\fe$ and $\fv$, and satisfies the properties given in \S \ref{S:tubular}.  (None exists for (B) since $\fs \cong \fsl(2,\bbC)$.)  The listed anti-involutions preserve $\fa$ in these cases, and since $\dim( N(\fa) / \fa) = 1$, the base curve for the tubular CR hypersurface model is affine homogeneous.  The classification (up to affine equivalence) of locally homogeneous curves in the affine plane consists of lines, quadrics, and the curves given below.
 
 \begin{tiny}
 \[
 \begin{array}{|c|c|c|} \hline
 \mbox{Real affine curve $F(x,u) = 0$}
 & \mbox{Classification} & 
 \begin{tabular}{c}
 CR syms of $F(\Re(z),\Re(w)) = 0$ beyond $i\partial_{z}, i\partial_{w}$
 \end{tabular}
 \\ \hline\hline
 \begin{array}{c} u = x^a\\
 (a \in \bbR \backslash \{ 0,\frac{1}{2},1,2 \}; \,\, a \sim \frac{1}{a})
 \end{array}
 & \mbox{(C)}, \, \gamma = \frac{a-2}{a-1},\, \varphi_1
 & z\partial_z + a w\partial_w \\ \hline
 u = x\ln(x) 
 & \mbox{(D)} 
 & z\partial_z + (z+w)\partial_w \\ \hline
 \begin{array}{c}
 x^2 + u^2 = \exp(b \arg(x+iu)) \\
 (b \in \bbR; \,\,\, b \sim -b)
 \end{array} 
 & \mbox{(C)},\, \gamma = \frac{3}{2} \pm \frac{b}{4} i,\, \varphi_2 
 & (b z - 2w)\partial_z + (b w + 2 z) \partial_w \\ \hline
 \end{array}
 \]
 \end{tiny}
 
 Upon complexification (\S \ref{S:Complexification}), $u = x^a$ and $u = x\ln(x)$ lead to the CR models underlying (C) (with $\gamma = \frac{a-2}{a-1} \in \bbR \backslash \{ 0,1,2,3\}$ and anti-involution $\varphi_1$) and (D).  This is not so straightforward for $x^2 + u^2 = \exp(b\arg(x+iu))$, so we instead work abstractly and align the Lie algebra data.  Letting
 \[
 L_1 = (bz - 2w) \partial_z + (bw+2z) \partial_w, \quad
 L_2 = i\partial_z, \quad L_3 = i\partial_w,
 \]
 we have
 \[
 [L_1,L_2] = -b L_2 - 2 L_3, \quad [L_1, L_3] = 2 L_2 - b L_3, \quad [L_2,L_3] = 0.
 \]
 On the other hand, the anti-involution $\varphi_2$ from case (C) has real fixed point set spanned by
 \[
 E_1 = 4 i (e_1 - e_2), \quad E_2 = e_2 + e_3, \quad E_3 = i(e_2 - e_3).
 \]
 Since we must have $\gamma = \frac{3}{2} + ti$, then
 \[
 [E_1,E_2] = 4t E_2 - 2 E_3, \quad
 [E_1,E_3] = 2 E_2 + 4t E_3, \quad [E_2,E_3] = 0.
 \]
 To align the structures, take $t = -\frac{b}{4}$.  Clearly $(x,u) \mapsto (x,-u)$ induces the parameter redundancy $b \mapsto -b$, so $\gamma = \frac{3}{2} \pm \frac{b}{4} i$.  Hence, all CR structures underlying (C) and (D) have been accounted for.
 
 All ODE models for (B) admit $\fs = \fsl(2,\bbC) \cong \fso(3,\bbC)$ symmetry.  We show that all underlying CR structures are exhausted by the Cartan hypersurfaces $(z,\bar{z}) = \beta|(z,z)|$ as in \eqref{E:Cartan-hyp}.  Since the classified anti-automorphisms for (B) are very complicated, we proceed in a different manner.  Note that $\fs$ has two real forms: $\fsl(2,\bbR) \cong \fso(2,1)$ and $\fsu(2) \cong \fso(3)$.  For each, a left-invariant CR structure is uniquely determined by a two-dimensional subspace $C$ with $[C,C] \not\subset C$, and a complex structure $J\colon C\to C$. The subspace $C$ is uniquely determined by its Killing perp $C^{\perp}$. In $\fsu(2)$, $[C,C] \not\subset C$ always.  In $\fsl(2,\bbR)$, $[C,C]\subset C$ if and only if $C^{\perp}$ is spanned by a nilpotent element, and $C$ is conjugate to the subalgebra of upper-triangular matrices in $\fsl(2,\bbR)$.  (Exclude this last case.)
	
Identifying $J$ with a $2\times 2$ real matrix satisfying $J^2 = -1$, we have $J^2 - \tr(J) J + \det(J) =0$.  If $\tr(J) \neq 0$, then $J$ is a real scalar matrix, which cannot satisfy $J^2 = -1$.  Thus, $\tr(J) = 0$ and $\det(J) = 1$.  Hence, $J = \left(\begin{smallmatrix} a & b \\ c & -a \end{smallmatrix}\right)$ with $a^2+bc=-1$, i.e.\ $c = -\frac{a^2+1}{b}$.
	
Let us classify all $J\colon C \to C$ up to automorphisms stabilizing $C$.  For example, in $\fsl(2,\bbR)$ consider the case of $C^\perp$ spanned by $h = \left(\begin{smallmatrix} 1 & 0 \\ 0 & -1 \end{smallmatrix}\right)$. Then $C$ has basis $e_1 = \left(\begin{smallmatrix} 0 & 1 \\ -1 & 0 \end{smallmatrix}\right)$, $e_2 = \left(\begin{smallmatrix} 0 & 1 \\ 1 & 0 \end{smallmatrix}\right)$, and $\Ad_{\exp(th)}$ is represented in this basis by $T= \left(\begin{smallmatrix}\cosh(2t) & \sinh(2t) \\ \sinh(2t) & \cosh(2t) \end{smallmatrix}\right)$.  With $J$ as above, we find that $TJT^{-1}$ is anti-diagonal iff $\tanh(4t) = \frac{2a}{b-c} = \frac{2ab}{a^2+b^2+1}$.  This is valued in $(-1,1)$, so such $t$ exists.  Thus, $J$ can always be brought to the form $\left(\begin{smallmatrix}0 & -\alpha \\ 1/\alpha & 0 \end{smallmatrix}\right)$.  Note that $C$ is also stable under $\Ad \left(\begin{smallmatrix} 0 & 1 \\ 1 & 0 \end{smallmatrix}\right)$. It induces the transformation $e_1\mapsto -e_1$, $e_2\mapsto e_2$ and thus the parameter equivalence $\alpha \sim -\alpha$.  Performing similar computations for other subspaces $C$, we get the following list of all algebraic models for case~(B).

\begin{tiny}
	\[
	\begin{array}{|c|c|c|c|c|} \hline
	\mbox{Real form of $\fsl(2,\bbC)$} & \mbox{Basis of $C^\perp$} 
	& \mbox{Basis of the contact plane $C$} & \mbox{Complex structure}
	& \mbox{Parameter equivalence} 
	\\ \hline\hline
	\fsu(2) & \begin{pmatrix} i & 0 \\ 0 & -i \end{pmatrix} 
	& e_1 = \begin{pmatrix} 0 & 1 \\ -1 & 0 \end{pmatrix}, \,\,
	e_2 = \begin{pmatrix} 0 & i \\ i & 0 \end{pmatrix} & 
	\begin{pmatrix} 0 & -\alpha \\ 1/\alpha & 0 \end{pmatrix}
	& \alpha \sim -\alpha, 1/\alpha
	\\ \hline
	\fsl(2,\bbR) & \begin{pmatrix} 0 & 1 \\ -1 & 0 \end{pmatrix} 
	& e_1 = \begin{pmatrix} 1 & 0 \\ 0 & -1 \end{pmatrix}, \,\,
	e_2 = \begin{pmatrix} 0 & 1 \\ 1 & 0 \end{pmatrix} & 
	\begin{pmatrix} 0 & -\alpha \\ 1/\alpha & 0 \end{pmatrix}
	& \alpha \sim -\alpha, 1/\alpha
	\\ \hline
	\fsl(2,\bbR) & \begin{pmatrix} 1 & 0 \\ 0 & -1 \end{pmatrix} 
	& e_1 = \begin{pmatrix} 0 & 1 \\ -1 & 0 \end{pmatrix}, \,\,
	e_2 = \begin{pmatrix} 0 & 1 \\ 1 & 0 \end{pmatrix} & 
 	\begin{pmatrix} 0 & -\alpha \\ 1/\alpha & 0 \end{pmatrix}
	& \alpha \sim -\alpha
	\\ \hline
	\end{array}
	\]
\end{tiny}
The first of these cases is treated in detail in~\cite{Cap2006}. The two others are their non-compact analogues.

To construct the local models, we proceed as in \S \ref{S:Cartan-hyp} and just give a summary here.  Given $z = (z_1,z_2,z_3)^\top$, consider $(z,z) = \sum_{j=1}^3 \epsilon_j (z_j)^2$, where $\epsilon_j = \pm 1$.  Then $\fso(3,\bbC)$ is identified with the $\bbC$-span of $\sfZ_{jk} = \epsilon_j z_j \partial_{z_k} - \epsilon_k z_k \partial_{z_j}$, where $1 \leq j < k \leq 3$.  Let $\cQ = \{ [z] : (z,z) = 0 \}$. 
Given $[z]$ in a real 3-dimensional orbit of $\tO(3)$ or $\tO(2,1)$ in $\bbC\bbP^2 \backslash \cQ$, $v = \Re(z)$ and $w = \Im(z)$ span a real 2-plane $\Pi$ in $\bbR^3$.   Using complex multiplication, we can assume $(v,w) = 0$.  Taking $(\cdot,\cdot)|_\Pi$ non-degenerate, the following are normal forms for $[z] \in \bbC\bbP^2 \backslash \cQ$ (which slightly differ from those given in \S \ref{S:Cartan-hyp}):
 \begin{enumerate}
 \item $\Pi$ is positive-definite and the signature of the scalar product $(z,z)$ is $(+++)$ or $(++-)$: $z = (1,iy,0)$, where $0 < y < 1$.  Then $\beta = \frac{1+y^2}{1-y^2} > 1$.
 \item $\Pi$ is indefinite: assuming signature $(++-)$, we have: $z = (1,0,iy)$, where $0 < y \ne 1$.  Then $\beta = \frac{1-y^2}{1+y^2}$ satisfies $0< |\beta| < 1$.
 \end{enumerate}
Matching these orbits with the above algebraic models is straightforward.  Fix the affine chart $(1,z_2,z_3)$ in $\bbC\bbP^2$.  For case (1), the orbit of $\tO(3)$ or $\tO(2,1)$ has 3-dimensional real tangent space in $\bbC^2 \cong T_{[z]}\bbC\bbP^2$ given by $\langle (1-y^2)\partial_{z_2}, \partial_{z_3}, iy \partial_{z_3}\rangle$.  (The scaling $z_1 \partial_{z_1} + z_2 \partial_{z_2} + z_3 \partial_{z_3}$ on $\bbC^3$ induces a trivial action on $\bbC\bbP^2$, so in the given chart, we can make the substitutions $\partial_{z_1} = -z_2 \partial_{z_2} - z_3 \partial_{z_3}$ into $\sfZ_{jk}$.)  Then $C = \langle \partial_{z_3}, iy \partial_{z_3} \rangle$, and multiplication by $i$ is represented by $\left(\begin{smallmatrix} 0 & -y \\ 1/y & 0 \end{smallmatrix}\right)$ in this basis.  The range $0 < y < 1$ matches with the first two cases in the table above (namely, set $y=\alpha$).  Case (2) is handled similarly.

\bibliographystyle{amsplain}

\end{document}